\documentclass[11pt,a4paper]{article}

\usepackage[english]{babel}
\usepackage{pst-grad} 
\usepackage{pst-plot} 
\usepackage{pstricks}
\usepackage{amsmath,amssymb,mathrsfs,amsthm}
\usepackage{tikz} 
\usepackage{graphicx}

\newcommand{\FF}{\mathcal F}

\newcommand{\RR}{\mathbb R}

\newcommand{\ZZ}{\mathbb Z}

\newcommand{\TT}{\mathbb T}

\def\jump#1{{[\hspace{-2pt}[#1]\hspace{-2pt}]}}
\def\norm#1{\left\vert \left\vert #1  \right\vert \right\vert }
\def\seminorm#1{ \left\vert #1  \right\vert }

\newcommand{\St}{\mathcal{S}}

\def\tnorm#1{\left\vert \left\vert \left\vert #1  \right\vert \right\vert \right\vert}

\usepackage[bookmarks,colorlinks]{hyperref}
\hypersetup{colorlinks=true, linkcolor = blue,
	anchorcolor = blue,
	citecolor = blue,
	filecolor = blue,
	urlcolor = blue}

\newtheorem{lem}{Lemma}
\newtheorem{prop}{Proposition}
\newtheorem*{prop-non}{Proposition}
\newtheorem{theo}{Theorem}
\newtheorem*{theorem-non}{Theorem}

\usepackage[a4paper]{geometry}
\title{Long time interface dynamics for gravity Stokes flow}
\author{
	Francisco Gancedo
	\\{\footnotesize Departamento de An\'alisis Matemático \& IMUS}
	\\{\footnotesize Universidad de Sevilla}
	\\{\footnotesize Sevilla, Espa\~na}
	\\{\footnotesize email: {\it fgancedo@us.es}}
	\and 
	Rafael Granero-Belinch\'on
	\\{\footnotesize Departamento de Matem\'aticas, Estad\'istica y Computaci\'on}
	\\{\footnotesize Universidad de Cantabria}
	\\{\footnotesize Santander, Espa\~na}
	\\{\footnotesize email: {\it rafael.granero@unican.es}}
	\and 
	Elena Salguero
	\\{\footnotesize ICMAT \& IMUS}
	\\{\footnotesize Madrid \& Sevilla, Espa\~na}
	\\{\footnotesize email: {\it esalguero@us.es}}
}

\begin{document}
\maketitle 

\begin{abstract}
	We study the dynamics of the interface given by two incompressible viscous fluids in the Stokes regime filling a 2D horizontally periodic strip. The fluids are subject to the gravity force and the density difference induces the dynamics of the interface. We derive the contour dynamics formulation for this problem through a  $x_1$-periodic version of the Stokeslet. Using this new system, we show local-in-time well-posedness when the initial interface is described by a curve with no self-intersections and $C^{1+\gamma}$ Hölder regularity, $0<\gamma<1$. This well-posedness result holds regardless of the Rayleigh-Taylor stability of the physical system. In addition, global-in-time existence and decay to the flat stationary state is proved in the Rayleigh-Taylor stable regime for small initial data. Finally, in the Rayleigh-Taylor unstable regime, we construct a wide family of smooth solutions with exponential in time growth for an arbitrary large interval of existence. Remarkably, the initial data leading to this exponential growth possibly lack any symmetry.
\end{abstract}

{\small
\tableofcontents}
\section{Introduction}
In this paper, we study the long time dynamics of the interface separating two different fluids in a regime with very low Reynolds number. The motion of a two-phase highly viscous flow is a practical problem in many applications ranging from Biology to Physics. Furthermore, this is also a first principles step towards the description of viscous water waves. Indeed, although for most of the applications in coastal engineering, water waves are assumed to be inviscid, as already stated by the celebrated applied mathematician Longuet-Higgins \cite{longuet1992theory} (see \cite{granero2022interfaces} for more details regarding the mathematical modeling of viscous waves):
\begin{quote}
	\emph{For certain applications, however, viscous damping of the waves is important, and it would be highly convenient to have equations and boundary conditions of comparable simplicity as for undamped waves.}
\end{quote}

With this physical phenomena in mind, we consider two incompressible, viscous and immiscible fluids filling the  $2\pi$-periodic strip in the variable $x_1$, $\TT \times \RR$.  
The curve 
$$\Gamma(t)=\{ (z_1(\alpha,t),z_2(\alpha,t)); \quad \alpha \in [-\pi,\pi], \quad z(\alpha+ 2 \pi k,t) = (2 \pi k,0) + z(\alpha,t)  \},$$
is the interface between the fluids in such a way that
$$
\TT\times \RR=\Omega^{+}(t)\cup\Omega^{-}(t)\cup\Gamma(t),\quad \Omega^{+}(t)\cap\Omega^{-}(t)=\phi,\quad \partial\Omega^{\pm}(t)=\Gamma(t).
$$
Then the upper fluid fills the domain $\Omega^{+}(t)$, while the lower fluid lies in $\Omega^{-}(t)$. More precisely, there exists a constant $M$ big enough such that 
$$
\TT \times[M,+\infty)\subset\Omega^+(t)\quad\mbox{and}\quad
\TT \times(-\infty,-M]\subset\Omega^-(t).
$$
At low Reynolds numbers, the viscosity forces dominate the inertial forces acting on the fluid. In the limit case, where inertial forces are neglected, the fluids flow in the Stokes regime. There is a vast literature about the study of the Stokes operator and Stokes-based models. For instance we refer to the pioneer works by Ladyzhenskaya \cite{ladyzhenskaya_mathematical_1969}, Ladyzhenskaya \& Solonnikov \cite{ladyzenskaja_1980} and Abe \& Giga \cite{giga_acta}. In this paper, we consider the two incompressible fluids modeled by Stokes flow, hence the dynamics of the fluids is described by the following system of PDEs:
\begin{subequations}\label{stokes}
	\begin{align}
		-\mu^\pm\Delta u^\pm&=-\nabla p^\pm-g\rho^\pm(0,1)^t,&&\quad x\in\Omega^\pm(t),t\in[0,T],\\
		\nabla\cdot u^\pm&=0,&&\quad x\in\Omega^\pm(t),t\in[0,T],\\
		\jump{-pI + \mu ((\nabla u +(\nabla u)^T)/2)}\cdot(\partial_\alpha z)^\perp&= 0,&&\quad x\in\Gamma(t),t\in[0,T],\\
		\jump{u}&=0,&&\quad x\in\Gamma(t),t\in[0,T],\\
		z_t&= u(z,t), &&\quad  t\in[0,T],\\
		z&= z_0, &&\quad t =0 .
	\end{align}
\end{subequations}
Above we have used the notation
$$
(\partial_\alpha z)^\perp=(-\partial_\alpha z_2,\partial_\alpha z_1),
$$
$$
\jump{F}=F^+(z(\alpha,t),t)-F^-(z(\alpha,t),t),
$$
giving in (\ref{stokes}c) the continuity of the stress tensor and in (\ref{stokes}d) the continuity of the velocity due to the viscosity of the fluids.
In the previous equations, $g$ refers to the acceleration due to gravity while $u, p,\mu$ and $\rho$ denote the velocity field, pressure, dynamic viscosity and density of each fluid, respectively. Then, $z$ describes the internal wave separating both fluids, which moves with the flow. This internal wave is a free boundary and should be recovered from the dynamics of the problem. We assume that these fluids have the same viscosities but different densities, \emph{i.e.} 
$$
\mu^+=\mu^-,\quad \rho^+\neq\rho^-.
$$
In what follows, we assume, without loss of generality, that $\mu=g=1$. 

\begin{figure}[h]\label{fig1}
	\begin{center}
		\begin{tikzpicture}[domain=0:2*pi, scale=0.9]
			\draw[ultra thick, smooth, color=magenta] plot (\x,{1+0.5*cos(\x r)});
			\draw[very thick,<->] (2*pi+0.4,0) node[right] {$x_1$} -- (0,0) -- (0,2.5) node[above] {$x_2$};
			\node[right] at (2.5,1.2) {$\Omega^+(t), \, \, \rho^+$};
			\node[right] at (5,0.8) {$\Omega^-(t), \, \, \rho^-$};
		\end{tikzpicture}
	\end{center}
	\caption{The situation under study.}
\end{figure}
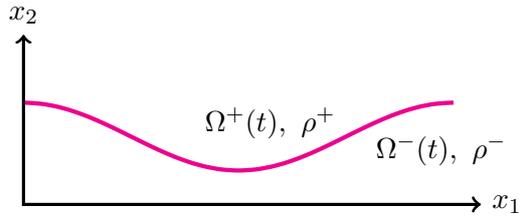

The dynamics of free boundary viscous waves driven by different forces is a classical and challenging problem, where many scenarios are still unsolved. The fluid dynamics can be subjected to the Stokes regime, Navier-Stokes regime (see \cite{guo_almost_2013}) or a water waves regime with viscosity (see \cite{granero-belinchon_well-posedness_2021}), among others. The fluids are typically driven by gravity and capillarity forces, and their dynamics are also subject to the geometry of the problem.

The two-phase Stokes regime driven by capillarity in $\RR^2$, where the interface between the two fluids is described by the graph of a function, was first studied by Badea \& Duchon \cite{badea_capillary_1998} in the contour dynamics formulation approach. In this work, the authors prove a global existence and uniqueness result for small initial data in the space of Fourier transforms of bounded measures. This scenario has also been studied by Matioc \& Prokert \cite{matioc_two-phase_2021,matioc_two-phase_2022}, with equal viscosities and viscosity jump, respectively. The authors prove well-posedness in sub-critical Sobolev spaces with arbitrary size initial data and a criterion for global existence, using the potential theory approach. The linear problem for a graph in the case of a Stokes flow driven by capillarity is
$$f^L_t =- \sigma(-\partial_x^2)^{1/2} f^L,\quad \sigma>0.$$
Taking advantage of the parabolic character of the capillarity problem, they also prove parabolic smoothing of the solution. The same authors have recently proven a similar result for the one-phase problem (i.e., the fluid above behaves like vacuum) as a small viscosity limit \cite{matioc_capillarity_2022}.  	
Moreover, Guo \& Tice \cite{guo2018stability} (see also \cite{zheng2017local}) study the stability of contact lines in fluids. In particular, they consider the dynamics of an incompressible viscous Stokes fluid evolving in a two-dimensional open-top vessel under the influence of gravity and capillary forces. Furthermore, the sedimentation of intertialess particles in a viscous flow in 3D is described by a coupled transport-Stokes system subject to gravity. This model was derived by Höfer \cite{hofer_2018_stokes} and Mecherbet \cite{mecherbet_2019_stokes}. Regarding this problem, Höfer \& Schubert \cite{hofer_2021_stokes} and Mecherbet \cite{mecherbet_2021_droplet}  prove in parallel works the existence and uniqueness of solutions with initial density in $L^1 \cap L^\infty$. Further results concerning the transport-Stokes system problem in 2D, are, for instance, \cite{free_boundary} and recent work by Grayer II \cite{grayer_ii_dynamics_2022}, where the authors show global well-posedness for compactly supported and $L^1 \cap L^\infty$ initial density and persistence of regularity results. Moreover, in the also recent result \cite{leblond}, Leblond proves global well-posedness for bounded initial density in the case of bounded domains and the infinite strip $\Omega =  \RR \times (0,1)$. See also \cite{hofer_2022_filaments}, where Stokes is used to model vortex filaments. A more exhaustive collection of results for the transport-Stokes system can be found in \cite{mecherbet_few_2022}.

In this paper we will use a contour dynamic approach to write the Stokes free boundary problem as a single nonlinear and non-local partial differential equation. The study of nonlinear and non-local PDEs have been a very active research area in Applied Mathematical Analysis in the last years. Two celebrated problems sharing the linear part of the capillary driven Stokes problem are the Peskin and the Muskat problems. Furthermore, both problems can be studied as a nonlinear and nonlocal partial differential equation using a contour dynamics approach similar to the one that we use in this paper. On the one hand, the Peskin problem models the dynamics of an elastic filament immersed in a Stokes fluid \cite{mori_wellposedness_2019,garcia-juarez_peskin_2020,lin2019solvability,chen_peskin_2021}. However, in the problem under consideration in this paper, the gravity driven Stokes problem, the dynamics is induced by the gravity force. On the other hand, the Muskat problem models the dynamics of fluids with different densities flowing in a porous media according to Darcy's law \cite{gancedo_survey_2017,granero2020growth}.

The results concerning the gravity driven free boundary Stokes problem are much more scarce when the fluids fill an infinitely deep domain. We start noting that the Stokes gravity interface dynamics when the internal wave is given as the graph of the function $h(x,t)$ verifies the following energy balance
$$
\frac{(\rho^--\rho^+)}{2}\|h(t)\|_{L^2(\TT)}^2+\int_{0}^t\|\nabla u\|_{L^2(\TT \times \RR)}^2ds=\frac{(\rho^--\rho^+)}{2}\|h(0)\|_{L^2(\TT)}^2.
$$
Furthermore, the Stokes gravity interface dynamics the linear operator reads
$$(-\partial_x^2)^{1/2}f^L_t =- (\rho^--\rho^+) f^L,$$
which is not of parabolic type. Instead, we show that this operator has a damping effect without regularization if the denser fluid is below the lighter one: a situation know as the Rayleigh-Taylor stable scenario. We call this a weak damping effect. This is in heavy contrast to the Peskin problem, the Muskat problem and the capillary driven Stokes problem where the linear operator is given by a square root of the Laplacian as shown before. This striking difference makes the gravity driven Stokes problem a hyperbolic problem for the free boundary. For this hyperbolic non-local problem we establish the local existence for $C^{1,\gamma}$ interfaces and global well-posedness and decay for sufficiently small free boundaries in the RT stable case. We provide two global existence theorems. First, we establish global well-posedness of classical solutions in Sobolev spaces and polynomial decay of the $L^2$ norm towards the flat interface. Secondly, we prove global well-posedness of analytic interfaces in Wiener algebras and exponential decay of the $A^0$ norm towards the flat interface. In the two-phase gravity driven Stokes, our result is, to the best of our knowledge, the only one dealing with an infinitely deep domain with non integrable densities, instead our densities are merely in $L^\infty$ with no decay.

The local well-posedness results extends to the Rayleigh-Taylor unstable regime, establishing the existence of local in time $C^{1,\gamma}$ unstable interfaces. Despite being well-posed, in the final result of this paper, we prove that interfaces in the unstable regime growth exponentially in some Wiener norms. We prove this result by a careful study of the linear semi-group, which shows exponential growth bounds. We can generalize this behavior to the full non-linear problem, showing exponential growth for smooth small data in Wiener spaces in an arbitrary large interval of existence. Previous results in this direction are scarce and sometimes are more restricted. In this direction, we have to mention the previous result \cite{Guo2007}, where Guo, Hallstrom \& Spirn showed instability for smooth interfaces in three different physical settings. Namely, vortex sheets with surface tension, the Muskat problem with surface tension and vortex patches. They show that the dynamics of perturbations of the steady states are characterized by the dynamics of their linear semi-groups in a particular time scale related to the size of the perturbation. We would also like to mention the result by Kiselev \& Li \cite{Kiselev2019}, where the authors show exponential growth of the curvature of an Euler patch boundary under certain symmetry assumptions and with fixed boundary. We have to emphasize that in our result, we do not need any symmetry or time scale restriction, instead we prove the exponential growth for arbitrary time scales and non-symmetric initial data.
	
\section{Main results and methodology}
One of the classical methods to deal with free boundary problems is to exploit potential theory in order to reformulate the problem into a new contour dynamics equation, which will be typically non-local and strongly non-linear. Let us mention that this kind of approach has been extensively and successfully used in other free boundary problems in fluid dynamics to show well-posedness (see \cite{bertozzi_global_1993,chemin_persistance_1993} for the vortex patch, \cite{cordoba_interface_2010} for water waves, \cite{kiselev2017local,gancedo2021local,gancedo_well-posedness_2021} for the SQG sharp-front, \cite{cordoba_interface_2011, gancedo_global_2021}  for the Muskat problem and \cite{garcia-juarez_peskin_2020,chen_peskin_2021} for the Peskin problem). Furthermore, it has been applied to prove singularity formation for the water waves, the SQG sharp-front and the Muskat problem \cite{castro2012rayleigh,castro2013finite,kiselev2016finite,gancedo2021local,castro2013breakdown}. In this work, we derive the contour dynamics formulation for \eqref{stokes} through a new kernel, which we call the $x_1$-periodic Stokeslet. This approach results in the following equivalent contour dynamics equation for the internal wave $z(\alpha,t)$: 
\begin{align} \label{eq:1}
	z_t(\alpha,t) & =  (\rho^- - \rho^+)  \int_{\TT} \St(z(\alpha,t)-z(\beta,t)) \cdot \partial_\beta z^\perp (\beta,t) z_2(\beta,t)  d\beta,
\end{align}
where the so-called $x_1$-periodic Stokeslet reads 
\begin{equation*}
	\St(y) =   \frac{1}{8 \pi } \log \left(2(\cosh (y_2)-\cos( y_1)) \right) \cdot I
	- \frac{y_2}{8 \pi (\cosh (y_2) - \cos (y_1))} \left( \begin{array}{cc}
		-\sinh (y_2) &\sin (y_1) \\
		\sin (y_1) & 	\sinh (y_2)
	\end{array} \right).
\end{equation*}

We take advantage of this contour dynamics formulation \eqref{eq:1} to prove the following local in time existence result for the interface with Hölder regularity $C^{1,\gamma}(\TT)$, in the case when the initial geometry of the interface is a curve $z_0(\alpha)$ satisfying the arc-chord condition (which is measured by the function $\mathcal{F}$ defined in \eqref{arcchord}-\eqref{arcchord0}).
\begin{prop-non}[Local existence of solutions in  $C^{1,\gamma}(\TT)$]
	Let $0<\gamma<1$ be a fixed parameter and $z_0(\alpha) \in C^{1,\gamma}(\TT)$ be the initial data satisfying the arc-chord condition.
	Then, there is a time $0<T$ such that there exists a unique local solution 
	$$
	z \in C^1((-T,T);C^{1,\gamma}(\TT))
	$$ of \eqref{eq:1} satisfying the arc-chord condition on $(-T,T)$.
\end{prop-non}
The proof of this result relies on the Picard Theorem on a suitable Banach space. Note that it is independent of the regime of the densities, being valid in both the stable and unstable case. Further details of the proof are shown in Section \ref{section:localexistence}.

In the particular case when the internal wave is described as the graph of a function $h(\alpha,t)$, i.e.,
$$z(\alpha,t) = (\alpha,h(\alpha,t)),$$
we show that the contour dynamics equation is equivalent to  
\begin{align}
	h_t(\alpha,t) &= \frac{(\rho^- - \rho^+)}{8 \pi} \int_\TT \log (2(\cosh (h(\alpha,t)-h(\beta,t))-\cos( \alpha-\beta))) h(\beta) [ 1+h_\alpha(\alpha) h_\alpha (\beta) ]   d \beta \nonumber\\
	&\quad+  \frac{(\rho^- - \rho^+)}{8 \pi} \int_\TT \frac{h(\beta) (h(\alpha)-h(\beta) )}{\cosh(h(\alpha)-h(\beta)) - \cos (\alpha-\beta)} \left[ (h_\alpha(\alpha)h_\alpha(\beta)-1) \sinh(h(\alpha)-h(\beta))\right] d\beta\nonumber\\
	&\quad+  \frac{(\rho^- - \rho^+)}{8 \pi} \int_\TT \frac{h(\beta) (h(\alpha)-h(\beta) )}{\cosh(h(\alpha)-h(\beta)) - \cos (\alpha-\beta)} \left[(h_\alpha (\alpha) + h_\alpha (\beta) ) \sin (\alpha-\beta)   \right] d\beta. \label{eq:2}
\end{align}
The detailed proofs of these formulations are given in Section \ref{section:contourdynamics}.

In order to prove global in time existence of solutions in Sobolev spaces in the stable regime of the densities, we study the linear semi-group 
$$f^L_t = -(\rho^--\rho^+)(-\partial_x^2)^{-1/2} f^L,$$
  which results from the linearization of \eqref{eq:2} around the equilibrium, and prove a decay result (see Section \ref{section:linearoperator} for further details). The linear semi-group properties together with a priori estimates in $L^2(\TT)$ and $H^3(\TT)$ lead us into one of the main results of this paper, where we show stability for classical solutions with prescribed small initial data in $H^3(\TT)$. Although we expect that the optimal regularity is less than $H^3$, we have chosen this space for the sake of simplicity.

\begin{theorem-non}[Global existence and decay of solutions for small data] 
	Let $h_0\in H^3(\TT)$ be the initial data for \eqref{eq:2}. Assume that  the system is in the RT stable regime, \emph{i.e.}
	\begin{equation*}
		\rho^- - \rho^+ >0,
	\end{equation*}
	and take $\frac{3}{2}<m<2$. There is a $0<\delta = \delta(\rho^- - \rho^+)$ such that if
	$$
	\norm{h_0}_{   H^{3}(\TT)}<\delta,
	$$ 
	there exists a unique global classical solution $h(\alpha,t)$ 
	$$
	h\in C([0,T];H^3),
	$$
	of \eqref{eq:2} for arbitrary $T>0$ such that
	$$
	(1+t)^{m} \norm{h}_{L^2}+ \norm{h}_{{  H}^3}\leq C \norm{h_0}_{H^3},	
	$$	
	for some constant $C >0$.
\end{theorem-non}
Further details of the proof are shown in Section \ref{section:globalexistence}.

Moreover, we prove global well-posedness in Wiener spaces of analytic functions $$A^1_{\nu}=\{u\in L^2(\mathbb{T})\text{ s.t. }\|u\|_{A^1_{\nu}}=\sum_{k=-\infty}^\infty e^{\nu |k|}|k|^{s}|\hat{u}(k)|<\infty\}$$
by a careful analysis of equation \eqref{eq:2} applying Fourier techniques and taking advantage of the algebra structure of Wiener spaces. This result is stated in the following theorem:
\begin{theorem-non}[Global existence of solutions in Wiener algebras for small data] 
	Let $h_0\in {A}_{\nu_0}^1(\TT)$ be the initial data for \eqref{eq:2} in the RT stable case,
	$$
\rho^- - \rho^+ >0.	
	$$
Assume that 
	\begin{equation*}
	\nu_0 >0.
	\end{equation*}
	There is a $0<\delta = \delta(\rho^- - \rho^+,\nu_0)$ such that if
	$$
	\norm{h_0}_{  {A}_{\nu_0}^1}<\delta,
	$$ 
	so that 
	$$ \nu_0 - \mathcal{M}(\norm{h_0}_{A_{\nu_0}^1}) >0,$$
	for a suitable non negative function $\mathcal{M}(x) \approx x + O(x^2)$,
	there exists a unique global analytic solution $h(\alpha,t)$ 
	$$
	h\in  C([0,T];{A}_{\nu(t)}^1(\TT)),
	$$
	of \eqref{eq:2}  for arbitrary $T>0$ satisfying
	$$
	\norm{h}_{  {A}_{\nu(t)}^1        }\leq \norm{h_0}_{A^1_{\nu_0}}. 	
	$$	
\end{theorem-non}
The details of the proof are shown in Section \ref{proof1}. This result, together with a careful study of the linear semi-group growth and the ideas of the proof of the global existence result in Sobolev spaces, allow us to prove exponential decay of solutions in Wiener algebras. The statement reads as follows:
\begin{theorem-non}[Global existence and exponential decay of solutions for small data] 
	Let $h_0\in A^1_{\nu_0}(\TT)$ be the initial data for \eqref{eq:2} in the RT stable case,
	$$
\rho^- - \rho^+ >0,	
	$$ fulfilling the hypotheses of Theorem \ref{globalwiener}. In particular,
	$$ \norm{h_0}_{A^1_{\nu_0}} < \delta,$$
	for a suitable $\delta =\delta(\rho^- - \rho^+,\nu_0)$. Assume that 
	\begin{equation*}
		0<\nu^\ast\leq \frac{\nu_0}{24}.
	\end{equation*}
	There is a $0< \varepsilon =  \varepsilon(\nu^\ast, \rho^- - \rho^+,\norm{h_0}_{A^1_{\nu_0}}  )$ such that if
	$$
	\norm{h_0}_{ {A}^0_{\nu^\ast}  }<\varepsilon,
	$$ 
	there exists a unique global analytic solution $h(\alpha,t)$ 
	$$
	h\in C([0,T]; {A}^0_{\nu^\ast} (\TT)  ),
	$$
	of \eqref{eq:2} for arbitrary $T>0$ satisfying
	$$
	e^{\sqrt{\frac{\rho^- - \rho^+}{4} \nu^\ast t}}\norm{h}_{{A}^0}+ \norm{h}_{ {A}^0_{\nu^\ast}   }\leq C \norm{h_0}_{A^0_{\nu^\ast}},
	$$	
	for some constant $C>0$.
\end{theorem-non}
The proof is given in Section \ref{proof2}.

Finally, using the previous result, we can prove that smooth initial data in the RT unstable scenario can lead to exponential growth in a particular Wiener norm. 
\begin{theorem-non}[Exponential growth of solutions in the RT unstable case for small data] 
	Let $T>0$ be an arbitrary fixed parameter. Then, it exists a family of smooth initial data 
	$$ g_0 \in A^0_{\nu\ast} (\TT),  $$
	such that 
	$$ g \in C([0,T]; A^0_{\nu\ast} (\TT) ),  $$
	is a solution of \eqref{eq:2} in the RT unstable regime,
	$$
	\rho^--\rho^+<0,
	$$	
and 
	\begin{equation} 
		\norm{g(\tau)}_{A^0_{\nu^\ast}}  \ge \frac{1}{C}	e^{\sqrt{\frac{|\rho^+-\rho^-|}{4}\nu^\ast \tau}}	\norm{g_0}_{{A}^0}, \quad \tau \in [0,T].
	\end{equation}
\end{theorem-non}

\section{The contour dynamics formulation in $\TT \times \RR$} \label{section:contourdynamics}
In this section we are going to find that system \eqref{stokes} is equivalent to a single non-local and nonlinear equation for the internal wave $z(\alpha,t)$. This new equation is called the contour dynamics formulation of \eqref{stokes}. Although such formulation for the Stokes problem has been used before \cite{badea_capillary_1998} in the case of $\mathbb{R}^2$, the fact that we are considering periodic waves makes the form of the Stokeslet different and its computation far from being trivial. Thus, as the closed form of the periodic stream function and the periodic Stokeslet for a given generic force $F$ can be of its own interest, we include some of the details in the next section. Then, we restrict us to the case where the force $F$ reduces to the gravity force to further simplify the integral kernels. 

\subsection{The stream function}
Let us consider a generic force $F$. The purpose of this section is to find a representation formula for the stream function $\psi$, \emph{i.e.} the function such that
\begin{equation}
u(x,t)=\nabla^\bot\psi(x,t)=(-\partial_{x_2} \psi(x,t),\partial_{x_1} \psi(x,t)),
\end{equation}
where $u$ is the velocity field given by \eqref{stokes}.
We have 
$$
\nabla^\bot\cdot u(x,t)=\Delta \psi(x,t).
$$
Let us consider the Stokes equations written as
\begin{align*}
\Delta u-\nabla p&=F,\quad x\in \TT \times \mathbb{R}\\
\nabla\cdot u&=0,\quad x\in \TT\times \mathbb{R}.
\end{align*}
Using the stream function, we find that 
\[
\Delta^2 \psi = \nabla^\perp \cdot F(x).
\]
Then, in order to solve the Stokes system of PDEs, we have to first solve the bilaplacian in the strip $\TT \times \RR$. This latter problem will be solved by finding the Green's function associated to this fourth order elliptic problem $K$. Namely, we will prove the existence of the stream function by finding 
\begin{equation}
\psi(x) = \int_{\TT \times \RR} K(x-y) \nabla^\perp \cdot F(y) dy,
\end{equation} 
where $K$ is any function satisfying
$$
\Delta^2 K=\delta,\text{ $x$ in $\mathbb{T}\times\mathbb{R}$}.
$$
As this Green function can be useful in many different problems in viscous fluid dynamics, we give its precise expression and the proof of its construction in the next proposition.
\begin{prop}[Green function of the $x_1$-periodic bilaplacian]
Define
	\begin{equation}
	K(x) = \sum_{n \in \ZZ} \beta_n(x_2) e^{in x_1} .
	\end{equation}
	where 
	\begin{align}\label{beta0}
	\beta_0(x_2) &=  \frac{ |x_2|x_2^2}{24\pi}. 	
	\\\label{betan}
	\beta_n(x_2) & = \frac{(|nx_2|+1)e^{-|n x_2|}}{8\pi n^2|n|} \text{ for }n \neq 0.
	\end{align}
Then $K$ verifies
$$
\Delta^2K=\delta,\text{ $x$ in $\mathbb{T}\times\mathbb{R}$}.
$$
\end{prop}
\begin{proof}
Above equation can be expressed as the product of two deltas	
\begin{equation}
\Delta^2 K(x) = \delta(x_2)\delta(x_1),\quad x_1\in\TT,\quad x_2\in\RR.
\end{equation}
The delta function in the periodic setting is given by  
\[
\delta(x_1) = \frac{1}{2 \pi} \sum_{n \in \ZZ} e^{in x_1},
\]
if it is understood in the distributional sense. Then, it yields 
\begin{equation} \label{bilap1}
\Delta^2 K(x) = \frac{\delta(x_2)}{2 \pi} \sum_{n \in \ZZ} e^{in x_1}.
\end{equation}
On the other hand, $K$ must be a $2\pi-$periodic function in $x_1$ so that we find $K(x)$ given by
\begin{equation} \label{bilap2}
K(x) = \sum_{n \in \ZZ} \beta_n(x_2) e^{in x_1} .
\end{equation}
Comparing \eqref{bilap1} and \eqref{bilap2}, we get an ODE for $\beta_n(x_2)$
\begin{equation}
\frac{\delta(x_2)}{2\pi} = \partial_{x_2}^4\beta_n(x_2)-2n^2 \partial_{x_2}^2\beta_n(x_2) + n^4 \beta_n(x_2).
\end{equation}
The case $n=0$ can be solve in a direct manner getting \eqref{beta0}. The case $n\neq 0$ can be solve by applying the Fourier Transform in the $x_2$ variable
\[
\hat f(\xi) = \int_{\RR} f(x_2) e^{-i x_2 \xi}dx_2.
\]
Then we find that
\begin{equation}
\frac{1}{ 2\pi} =\xi^4 \hat\beta_n(\xi) +2n^2 \xi^2 \hat\beta_n(\xi) + n^4 \hat\beta_n(\xi).
\end{equation}
We apply the inverse Fourier Transform to obtain
\begin{equation}
\beta_n(x_2) =  \frac{1}{(2 \pi)^2}\int_{\RR} \frac{1}{\xi^4 +2n^2 \xi^2 + n^4 } e^{i \xi x_2} d \xi
\end{equation}
Solving the above integral it is possible to get \eqref{betan}. This concludes the proof of the proposition.
\end{proof}
Let us further simplify the expression for the stream function. Integrating by parts,
\begin{align*}
\psi(x) &= \int_{\TT \times \RR} K(x-y) (-\partial_{y_2} F_1(y) + \partial_{y_1}F_2(y)) dy \\
& = \int_{\TT \times \RR}- \partial_{x_2} K(x-y) F_1(y) +   \partial_{x_1} K(x-y) F_2(y) dy \\
 &= \int_{\TT \times \RR} \Psi(x-y) \cdot F(y) dy,
\end{align*}
with
$$
\Psi=(- \partial_{x_2}K,\partial_{x_1} K),
$$
with
\begin{align*}
\partial_{x_1} K(x) &= \sum_{n \in \ZZ} i n \beta_n(x_2) e^{inx_1}\\
& = \frac{1}{8 \pi}\sum_{n \in \ZZ, n \neq 0} \left( \frac{i|x_2|}{n}e^{-|nx_2|}+ \frac{i}{n |n|}e^{-|nx_2|}\right) e^{in x_1} \\
& = - \frac{1}{4 \pi} \sum_{n > 0} \left( \frac{n|x_2|}{n^2}e^{-n|x_2|}+ \frac{1}{n^2}e^{-n|x_2|} \right) \sin(n x_1),
\end{align*}
and
\begin{align*}
 \partial_{x_2} K(x) &= \sum_{n \in \ZZ} \partial_{x_2}\beta_n(x_2) e^{inx_1}\\
& = \frac{x_2|x_2|}{8 \pi} -\frac{x_2}{8 \pi} \sum_{n \in \ZZ, n \neq 0} \frac{1}{|n|} e^{-|nx_2|} e^{in x_1} \\
 & =  \frac{x_2|x_2|}{8 \pi} -\frac{1}{4 \pi}   \sum_{n>0} nx_2\cos(nx_1)\frac{e^{-n |x_2|}}{n^2}.
 \end{align*}
We observe that, due to the Fourier series of the kernel $K$, we have that $\Psi$ is, at least, a function whose Fourier series is absolutely convergent and, as a consequence a $C(\mathbb{T}\times\mathbb{R})$ function. Therefore we can ensure the existence of the continuous function $\Psi$.

\subsection{The $x_1$-periodic Stokeslet}
The aim of this section is to use the representation formula for the stream function in order to obtain a representation formula for the velocity $u$. We recall that
\begin{align*}
\psi(x) &= \int_{\TT \times \RR} \Psi(x-y) \cdot F(y) dy,
\end{align*}
with
$$
\Psi=(- \partial_{x_2} K,\partial_{x_1} K),
$$
and $F$ is a generic force. As a consequence,
\begin{align*}
u(x)&= \nabla^\perp\psi(x)\\
 &= \int_{\TT \times \RR} \St(x-y) \cdot F(y) dy,
\end{align*}
where the $x_1$-periodic kernel is
$$
\St=\left(\begin{array}{cc}\partial_{x_2}^2 K &-\partial_{x_2}\partial_{ x_1} K\\ -\partial_{x_2}\partial_{ x_1} K & \partial^2_{ x_1} K\end{array}\right),
$$
with
\begin{align*}
\partial_{x_1}^2 K(x) &= - \frac{1}{4 \pi} \sum_{n > 0} \left( |x_2|e^{-n|x_2|}+ \frac{1}{n}e^{-n|x_2|} \right) \cos(n x_1),
\end{align*}
\begin{align*}
\partial_{x_1}\partial_{x_2} K(x) & =   \frac{1}{4 \pi}   \sum_{n>0} x_2\sin(nx_1)e^{-n |x_2|},
 \end{align*}
and
 \begin{align*}
\partial_{x_2}^2 K(x) & =  \frac{|x_2|}{4 \pi} -\frac{1}{4 \pi}   \sum_{n>0} \left(\frac{e^{-n |x_2|}}{n}- |x_2|e^{-n |x_2|}\right)\cos(nx_1).
 \end{align*}
We observe that, even if the kernel $\Psi$ was not explicit, we are able to sum up the Fourier series corresponding to the Stokeslet to find a closed form expression for the kernel $\St$. We compute
\begin{align*}
\sum_{n>0}  \sin(nx_1)e^{-n |x_2|} & =\sum_{n>0} \frac{e^{n(ix_1-|x_2|)}-e^{n(-ix_1-|x_2|)}}{2i}\\
&=   \frac{1}{2i}\Big( \frac{e^{ix_1-|x_2|}  }{1-e^{ix_1-|x_2|}} -\frac{e^{-ix_1-|x_2|}  }{1-e^{-ix_1-|x_2|}} \Big) \\
& =  \frac{1}{2}\frac{\sin (x_1)}{\cosh (x_2)-\cos (x_1)}=  \frac{1}{2}\partial_{x_1}\log(\cosh (x_2)-\cos (x_1)).
\end{align*}
As a consequence
\begin{equation*}
\sum_{n>0}  \cos(nx_1)\frac{e^{-n |x_2|}}{n}  = C(x_2) -\frac{1}{2} \log(\cosh (x_2)-\cos (x_1)).
\end{equation*} 
Similarly,
\begin{align*}
\sum_{n>0}  &\cos(nx_1)e^{-n |x_2|}|x_2| = \sum_{n>0} \frac{e^{n(ix_1-|x_2|)}+e^{n(-ix_1-|x_2|)}}{2}|x_2|\\
&= \frac{1}{2}\left( \frac{e^{ix_1-|x_2|}  }{1-e^{ix_1-|x_2|}} +\frac{e^{-ix_1-|x_2|}  }{1-e^{-ix_1-|x_2|}} \right)|x_2|= \frac{1}{2}  \frac{\cos (x_1) - e^{-|x_2|} }{\cosh(x_2)-\cos (x_1)} |x_2|\\
& =  \frac{1}{2}  \frac{\cos (x_1)-\cosh(x_2)+\cosh(x_2) - e^{-|x_2|} }{\cosh(x_2)-\cos (x_1)}|x_2|=- \frac{1}{2}|x_2|+\frac{1}{2}  \frac{\sinh(|x_2|) }{\cosh(x_2)-\cos (x_1)}|x_2| \\
& = -\frac{1}{2}|x_2|+\frac{1}{2}  \frac{x_2\sinh(x_2) }{\cosh(x_2)-\cos (x_1)}.
\end{align*} 
Thus, we find that
\begin{align*}
\partial_{x_2} \sum_{n>0}  \cos(nx_1)\frac{e^{-n |x_2|}}{n} & = \frac{1}{2}\frac{x_2}{|x_2|}-\frac{1}{2}  \frac{\sinh(x_2) }{\cosh(x_2)-\cos (x_1)}\\
&= C'(x_2) -\frac{1}{2} \frac{\sinh(x_2)}{\cosh (x_2)-\cos (x_1)},
\end{align*} 
from where
$$
\sum_{n>0}  \cos(nx_1)\frac{e^{-n |x_2|}}{n}  =\frac{|x_2|}{2} -\frac{1}{2} \log(\cosh (x_2)-\cos (x_1))+c.
$$
Choosing $x_2=0$ and $x_1=\pi$, we find
$$
\sum_{n>0}  \cos(nx_1)\frac{e^{-n |x_2|}}{n} =\frac{|x_2|}{2} -\frac{1}{2} \log(\cosh (x_2)-\cos (x_1))-\frac{\log(2)}{2}.
$$
The previous computations lead to
\begin{align*}
 \partial_{y_1} \partial_{y_2} K(y) &=  \frac{y_2}{8 \pi} \frac{\sin (y_1)}{\cosh (y_2) - \cos (y_1)},\\
 \partial_{y_1}^2  K(y) & = - \frac{y_2}{8\pi} \frac{\sinh (y_2)}{\cosh( y_2)-\cos (y_1)} + \frac{1}{8 \pi} \log (2(\cosh (y_2)-\cos( y_1))),\\
 \partial_{y_2}^2 K(y) & =  \frac{1}{8 \pi} \log (2(\cosh (y_2)-\cos( y_1)))  + \frac{y_2}{8 \pi} \frac{\sinh (y_2)}{\cosh (y_2)-\cos (y_1)}.
\end{align*}
Collecting all the terms we can write
\[
\St(y) =   \frac{1}{8 \pi } \log \left( 2(\cosh (y_2)-\cos( y_1)) \right)\cdot I  - \frac{y_2}{8 \pi (\cosh (y_2) - \cos (y_1))} \left( \begin{array}{cc}
 -\sinh (y_2) & \sin (y_1) \\
\sin (y_1) & 	\sinh (y_2)
\end{array} \right).
\]
Similarly, we can compute the pressure using the Green function for the Poisson equation in $\TT \times \RR$, namely
\[
\mathcal{P}(x) =  -\frac{1}{4 \pi} \log \left( \cosh (x_2) - \cos (x_1)    \right).
\]
which solves
\[
-\Delta \mathcal{P} = \delta.
\]
Then, we find that
\begin{align*}
p(x) & =- \frac{1}{4 \pi} \int_{\TT \times \RR} \log \left( \cosh (x_2-y_2) - \cos (x_1-y_1)    \right) \nabla\cdot F(y) dy\\
& =  \frac{1}{4 \pi} \int_{\TT \times \RR} \frac{-\sin (x_1-y_1) F_1(y)-\sinh (x_2-y_2) F_2(y)}{ \cosh (x_2-y_2) - \cos (x_1-y_1) } dy    \\
& =  -\frac{1}{4 \pi} \int_{\TT \times \RR} \frac{(\sin (y_1),\sinh(y_2) ) \cdot  F(x-y)}{ \cosh (y_2) - \cos (y_1) } dy.
\end{align*}
\subsection{The equation for the free boundary: the case of an arbitrary curve}
In this section we are going to focus on the case where the force acting on the fluid $F$ is the gravity force in order to obtain the contour dynamics formulation for the free surface under consideration. Indeed, in our case of two homogeneous fluids separated by an internal wave under the action of gravity, we can write the acting force as 
$$
F(x,t)=g\rho(x,t) (0,1)=\nabla (\rho(x,t) x_2)\quad \mbox{for}\quad x\in\Omega^{\pm}(t),
$$ 
with
$$
\rho(x,t)=\left\{\begin{array}{cc}\rho^+ & \text{ for $x$ in $\Omega^+(t)$} \\ \rho^- & \text{ for $x$ in $\Omega^-(t)$} \end{array}\right..
$$
We can use this specific structure to further simplify the representation formulas in the situation under study. Let us start with the case of the stream function. We integrate by parts in order to find that
\begin{align*}
\psi(x) &= \int_{\TT \times \RR} \Psi(x-y) \cdot \nabla (\rho(y) y_2) dy\\
&=\int_{\Omega^+} \Psi(x-y) \cdot \nabla_y (\rho^+ y_2) dy+\int_{\Omega^-} \Psi(x-y) \cdot \nabla_y (\rho^- y_2) dy\\
&=-\int_{\Omega^+} \nabla_y\cdot\Psi(x-y) \rho^+ y_2 dy-\int_{\Omega^-} \nabla_y\cdot\Psi(x-y) \cdot \rho^- y_2 dy,\\
&\quad-\rho^+\int_{\TT} \Psi(x-z(\beta))\cdot (\partial_\beta z(\beta))^\perp  z_2(\beta) d\beta+\rho^-\int_{\TT} \Psi(x-z(\beta))\cdot (\partial_\beta z(\beta))^\perp  z_2(\beta) d\beta\\
& = (\rho^- - \rho^+)\int_{\TT} \Psi(x-z(\beta))\cdot (\partial_\beta z(\beta))^\perp z_2(\beta) d \beta.
\end{align*}
The previous computation implies that the velocity solves
\begin{align*}
u(x) & = (\rho^- - \rho^+)\int_{\TT} \St(x-z(\beta))\cdot z_\beta^\perp(\beta) z_2(\beta) d \beta,
\end{align*}
with
\begin{equation}\label{kernel}
\St(y) =   \frac{1}{8 \pi } \log \left(2(\cosh (y_2)-\cos( y_1)) \right) \cdot I
  - \frac{y_2}{8 \pi (\cosh (y_2) - \cos (y_1))} \left( \begin{array}{cc}
-\sinh (y_2) &\sin (y_1) \\
\sin (y_1) & 	\sinh (y_2)
\end{array} \right).
\end{equation}
Since the velocity is continuous across the interface we find the contour equation for the case of an arbitrary periodic curve
\begin{align}\label{eq:curvaarbitraria}
z_t(\alpha,t) & =  (\rho^- - \rho^+)  \int_{\TT} \St(z(\alpha,t)-z(\beta,t)) \cdot \partial_\beta z^\perp (\beta,t) z_2(\beta,t)  d\beta.
\end{align}

\subsection{The equation for the free boundary: the case of a graph}
In the previous section we have found the contour equation for the internal wave separating both fluids. However, such a free boundary is not necessarily parameterized as the graph of a function. In fact, the graph parametrization is not maintained by the contour equation \eqref{eq:curvaarbitraria}. In this section we are going to find the contour equation for the case where the free boundary is parameterized as a graph. In order to do that we start with an initial data given by the graph of certain function and we evolve such initial data with a \emph{modified contour dynamics equation}. Due to the smoothness of the evolution, the solution of this modified contour dynamics equation is a graph and, besides this, it is parameterized as a graph. Finally, we show that, due to the definition of our modified evolution, a reparametrization of the graph solves the contour dynamics equation \eqref{eq:curvaarbitraria} and as a consequence, the evolving graphs are equivalent to a solution of \eqref{eq:curvaarbitraria} after a reparametrization.

Suppose $z(\alpha,t)$ is a solution to the following \emph{modified contour dynamics equation}
\begin{equation}
z_t(\alpha,t) = (\rho^- - \rho^+)  \int_{\TT} \St(z(\alpha,t)-z(\beta,t)) \cdot \partial_\beta z^\perp (\beta,t) z_2(\beta,t)  d\beta+ \lambda(\alpha,t) \partial_\alpha z(\alpha,t),
\end{equation}
where $\lambda$ will be defined below. Assuming that the curve $z$ is initially a graph, i.e., 
$$
z_1(\alpha,0) = \alpha 
$$ 
then, at least for a short time $t$, we have $\partial_\alpha z_1(\alpha,t') \neq 0$ for every $t' \in [0,t]$. This way, taking
\begin{equation*}
\lambda (\alpha,t) =-(\rho^- - \rho^+) \left( \int_{\TT} \St(z(\alpha,t)-z(\beta,t)) \cdot \partial_\beta z^\perp (\beta,t) z_2(\beta,t)  d\beta \right)\cdot (1,0)(\partial_\alpha z_1(\alpha,t))^{-1},
\end{equation*}
we can reparametrize the curve $z$ as the graph of certain function $h(\alpha,t)$.

With this choice of $\lambda$, the equation for the graph becomes 
\begin{align}\label{eq:grafo}
 h_t(\alpha,t) &=  (\rho^- - \rho^+)  \int_{\TT} \St((\alpha-\beta,h(\alpha,t)-h(\beta,t))) \cdot (-\partial_\alpha h(\beta,t) ,1) h(\beta,t)  d\beta \cdot (-\partial_\alpha h(\alpha,t) ,1).
\end{align}
More explicitly,
\begin{align}\label{eq:grafo2}
 h_t(\alpha,t) &= \frac{(\rho^- - \rho^+)}{8 \pi} \int_\TT \log (2(\cosh (h(\alpha,t)-h(\beta,t))-\cos( \alpha-\beta)) ) h(\beta) [ 1+h_\alpha(\alpha) h_\alpha (\beta) ]   d \beta \nonumber\\
 &\quad+  \frac{(\rho^- - \rho^+)}{8 \pi} \int_\TT \frac{h(\beta) (h(\alpha)-h(\beta) )}{\cosh(h(\alpha)-h(\beta)) - \cos (\alpha-\beta)} \left[ (h_\alpha(\alpha)h_\alpha(\beta)-1) \sinh(h(\alpha)-h(\beta))\right] d\beta\nonumber\\
 &\quad+  \frac{(\rho^- - \rho^+)}{8 \pi} \int_\TT \frac{h(\beta) (h(\alpha)-h(\beta) )}{\cosh(h(\alpha)-h(\beta)) - \cos (\alpha-\beta)} \left[(h_\alpha (\alpha) + h_\alpha (\beta) ) \sin (\alpha-\beta)   \right] d\beta.
\end{align}
Then, define a reparametrization of the curve 
$$  {z}(\alpha,t) = \tilde z (\phi(\alpha,t),t)$$
with new parameter
$$
\xi = \phi(\alpha,t).
$$
We want to find $\phi$ such that $\tilde{z}(\xi,t)$ solves 
\begin{align*}
\tilde{z}_t(\xi,t) & =  (\rho^- - \rho^+)  \int_{\TT} \St(\tilde{z}(\xi,t)-\tilde{z}(\chi,t)) \cdot \partial_\chi \tilde{z}^\perp (\chi,t) \tilde{z}_2(\chi,t)  d\chi.
\end{align*}
Changing variables
$$
\chi=\phi(\beta),\quad \partial_\beta z^\perp (\beta,t)=\partial_\chi \tilde{z}^\perp (\chi,t) \phi'(\beta) ,\quad d\chi=\phi'(\beta)d\beta
$$
$$\int_{\TT} \St(\tilde{z}(\xi,t)-z(\beta,t)) \cdot \partial_\beta z^\perp (\beta,t) z_2(\beta,t)  d\beta=\int_{\TT} \St(\tilde{z}(\xi,t)-\tilde{z}(\chi,t)) \cdot \partial_\chi \tilde{z}^\perp (\chi,t) \tilde{z}_2(\chi,t)  d\chi
$$
On the one hand, we compute
\begin{align*}
z_t(\alpha,t)&=\partial_t \tilde{z}(\phi(\alpha,t),t)\\
& =  \phi_t(\alpha,t) \partial_\xi \tilde z(\xi,t) + \tilde z_t(\xi,t).
\end{align*}
Dealing with the non-local part of the contour equation we find
\begin{align*}
z_t(\alpha,t)& = (\rho^- - \rho^+)  \int_{\TT} \St(z(\alpha,t)-z(\beta,t)) \cdot \partial_\beta z^\perp (\beta,t) z_2(\beta,t)  d\beta+ \lambda(\alpha,t) \partial_\alpha z(\alpha,t)\\
& = (\rho^- - \rho^+)  \int_{\TT} \St(\tilde{z}(\xi,t)-\tilde{z}(\chi,t)) \cdot \partial_\chi \tilde{z}^\perp (\chi,t) \tilde{z}_2(\chi,t)  d\chi+ \lambda(\alpha,t)\phi_\alpha(\alpha,t)\partial_\xi \tilde{z}(\xi,t).
\end{align*}
As a consequence, if we chose $\phi$ solving
$$ \phi_t(\alpha,t) =  \lambda(z)(\alpha,t) \phi_\alpha(\alpha,t)$$
we obtain that $\tilde{z}$ solves
\begin{align*}
\tilde{z}_t(\xi,t) & =  (\rho^- - \rho^+)  \int_{\TT} \St(\tilde{z}(\xi,t)-\tilde{z}(\chi,t)) \cdot \partial_\chi \tilde{z}^\perp (\chi,t) \tilde{z}_2(\chi,t)  d\chi.
\end{align*}
Thus, finding a solution of \eqref{eq:grafo} is equivalent to finding a solution of \eqref{eq:curvaarbitraria}.

\section{The linear operator}\label{section:linearoperator} 
Let us linearize around the steady state $z_\ast = (\alpha,0).$ Then, we find that the linear problem reads
\begin{equation}\label{lambdaeq1}
h_t(\alpha,t)=  \frac{(\rho^- -\rho^+)}{4}\frac{1}{2\pi} \int_\TT h(\alpha-\beta,t) \log\left(4  \sin^2\left(\ \frac{\beta}{2}\right) \right) d \beta.
\end{equation}
Using the properties of the logarithm, we find the equation
\begin{equation}\label{lambdaeq}
h_t(\alpha,t)=  -\frac{(\rho^- -\rho^+)}{4} \Lambda^{-1} (h)(\alpha,t)+\frac{(\rho^- -\rho^+)}{4}\frac{\log\left(4 \right) }{2\pi} \int_\TT h(\beta,t) d \beta,
\end{equation}
where the operator $\Lambda^{-1}$ is given explicitly as follows
$$
\Lambda^{-1} (h)(\alpha,t)=-\frac{1}{2\pi} \int_\TT h(\alpha-\beta,t) \log \left( \sin^2\left(\ \frac{\beta}{2}\right) \right) d \beta.
$$
 Its Fourier coefficients are the following (see \cite{castro_uniformly_2016}) 
$$
\widehat{\Lambda^{-1}}(k)=\left\{\begin{array}{ll}
	-\frac{1}{2\pi} \int_\TT e^{-ik\beta} \log \big| \sin^2(\beta/2) \big| d \beta = \frac{1}{|k|}, & \text{ for }k \neq 0,\\
	-\frac{1}{2\pi} \int_\TT \log \big| \sin^2(\beta/2) \big| d \beta = \log (4),& \text{ for }k = 0.
\end{array}\right. 
$$
Another relevant singular operator is the Hilbert transform in the torus
\begin{equation}\label{hilbert}
	 \mathcal{H}(h)(\alpha,t) = \frac{1}{2 \pi} \text{PV} \int_{\TT} \cot \left(\frac{\beta}{2}\right) h(\alpha-\beta,t) d \beta
\end{equation}
which is related to the previous operator in the form
$$\partial_\alpha \left( \Lambda^{-1} (h)(\alpha,t)  \right)= - \mathcal{H}(h)(\alpha,t). $$

Let us emphasize that in the stable case, when $ \rho^- - \rho^+ >0$, the linear operator shows a damping effect without regularization. In this sense, the linear operator is \emph{not} of parabolic type. To simplify the notation we write
$$
\bar{\rho}=\frac{(\rho^- -\rho^+)}{4}.
$$
Using the previous computations, we have that
$$
\hat{h}_t(k,t) = \left\{\begin{array}{ll}
	-  \frac{\bar{\rho}}{|k|} \hat{h}(k,t), & \text{ for }k \neq 0,\\
	0,& \text{ for }k = 0,
\end{array}\right. 
$$
and it follows that 
\begin{align}\label{fouriermodes}
\hat{h}(k,t) =\left\{\begin{array}{ll}
e^{- \frac{\bar{\rho} }{|k|}t} \hat{h}_0(k),& \text{ for }k \neq 0,\\
\hat{h}_0(0), &  \text{ for }k=0.
\end{array}\right.
\end{align}
We will use the notation 
\begin{equation*}
h(x,t)= e^{- \bar{\rho} \Lambda^{-1} t} h_0(x),
\end{equation*}
for the semi-group acting on some zero-mean data $h_0$. We observe that if we set an initial data with zero mean, i.e., $\hat{h}_0(0,t) =0$, then, this property is conserved due to \eqref{fouriermodes}. In this section we will focus in studying the time decay of the solution to \eqref{lambdaeq} in the Rayleigh-Taylor stable case $\bar{\rho}>0$. 

Before stating our results, let us introduce the following Banach scale of spaces

\begin{equation}\label{sobolev}
	H^{s}_\nu=\{u\in L^2(\mathbb{T})\text{ s.t. }\sum_{k=-\infty}^\infty e^{2\nu |k|}|k|^{2s}|\hat{u}(k)|^2<\infty\},
\end{equation}

\begin{equation}\label{wiener}
A^{s}_\nu=\{u\in L^2(\mathbb{T})\text{ s.t. }\sum_{k=-\infty}^\infty e^{\nu |k|}|k|^{s}|\hat{u}(k)|<\infty\}.
\end{equation}
To simplify notation we write
$$
H^s=H^{s}_0,\quad H_{\nu}=H^{0}_\nu$$

together with
$$
A^s=A^{s}_0,\quad A_{\nu}=A^{0}_\nu, \quad A=A_0^0. $$


In particular, we have the following result:

\begin{prop}[Decay in $L^2$ and $A$]\label{PropDH} Let us consider equation \eqref{lambdaeq} in the RT stable case $\bar{\rho}>0$ with zero mean initial data $h_0$. Then we have that the solution verifies 
\begin{equation}  \label{lineardecay}
\norm{h}_{L^2} \leq C \norm{h_0}_{H^{s_0}} (1+t)^{-s},
\end{equation}
\begin{equation}  \label{lineardecayA}
	\norm{h}_{{A}} \leq C \norm{h_0}_{{A}^{s_0}} (1+t)^{-s},
\end{equation}
where
\begin{equation}
 \label{decayrate}
0< s <s_0
\end{equation}
is arbitrary. Furthermore, with initial analytic regularity, we also have the following exponential decay
\begin{equation}  \label{lineardecay2}
	\norm{h}_{{L}^2} \leq C \norm{h_0}_{{H}_\nu} e^{-\sqrt{\bar{\rho}\nu t}},
\end{equation}
\begin{equation}  \label{lineardecay2A}
	\norm{h}_{{A}} \leq C \norm{h_0}_{{A}_\nu} e^{-\sqrt{\bar{\rho}\nu t}}.
\end{equation}
\end{prop}
\begin{proof} Let us first obtain an algebraic decay rate for Sobolev initial data. In the RT stable case we fix $M>0<s_0$  and compute
\begin{align*}
	\norm{h}_{{L}^2}^2 & =  \sum_{|k| \leq M, k \neq 0 } |\hat{h}(k,t)|^2 + \sum_{|k| > M} |\hat{h}(k,t)|^2 \\
	& \leq   \sum_{|k| \leq M, k \neq 0 } e^{ -2  \frac{\bar{\rho}}{|k|} t}|\hat{h}_0(k)|^2 + \frac{1}{M^{2s_0}} \sum_{|k|>M } |\hat{h}_0(k)|^2 |k|^{2s_0} \\
	& \leq   e^{ -2 \bar{\rho} \frac{1}{M} t} \norm{h_0}_{{L}^2}^2 + \frac{1}{M^{2s_0}} \norm{h_0}_{ H^{s_0}}^2. 
\end{align*}
The global bound $x^ne^{-x}\leq C(n)$ for $x\geq 0$ provides
$$
\norm{h}_{L^2}^2  \leq  \frac{C(n)M^n}{\bar\rho^n 2^n t^n} \norm{h_0}_{L^2}^2 + \frac{1}{M^{2s_0}} \norm{h_0}_{{H}^{s_0}}^2
$$
where $n$ is arbitrary and $t>1$. Taking $M$ such that
\begin{equation*}
C(n) \frac{M^n}{ \bar\rho^n 2^n t^n} = \frac{1}{M^{2s_0}},
\end{equation*}
it is possible to get
\begin{align*}
\norm{h}_{ L^2} & \leq C(n,\bar\rho,s_0) \norm{h_0}_{ H^{s_0}} t^{-\frac{ s_0 n}{n+2 s_0}}.
\end{align*}
The uniform bound $\|h\|_{ L^2}\leq \norm{h_0}_{ L^2}$ for $t \leq 1$ yields the desired bound
\[
 \norm{h}_{ L^2} \leq 
 C(n,\bar\rho,s_0) \norm{h_0}_{ H^{s_0}} (1+t)^{-\frac{s_0 n}{n+2s_0}},
\]
increasing the constant $C(n,\bar\rho,s_0)$.
Taking $n$ big enough we find
\begin{equation*}
s:=\frac{ n s_0 }{n + 2 s_0}<s_0.
\end{equation*}
A similar approach yields
\[
\norm{h}_{ A} \leq 
C(n,\bar\rho,s_0) \norm{h_0}_{ A^{s_0}} (1+t)^{-\frac{s_0 n}{n+2s_0}}.
\]
We compute for a different $M>0$ the following
\begin{align*}
	\norm{h}_{ L^2}^2 & =  \sum_{|k| \leq M, k \neq 0 } |\hat{h}(k,t)|^2 + \sum_{|k| > M} |\hat{h}(k,t)|^2 \\
	& \leq  e^{ -2  \frac{\bar{\rho}}{M} t} \sum_{|k|\leq M, k \neq 0 } |\hat{h}_0(k)|^2 + e^{-2\nu M}\sum_{|k|>M } |\hat{h}_0(k)|^2 e^{2\nu|k|} \\
	& \leq   e^{ -2  \frac{\bar{\rho}}{M} t} \norm{h_0}_{ L^2}^2 + e^{-2\nu M} \norm{h_0}_{ H_\nu}^2. 
\end{align*}
Taking 
$$
\frac{\bar{\rho}}{M}t=\nu M,
$$
allows us to conclude that
$$
\norm{h}_{ L^2} \leq\norm{h_0}_{ H_\nu} e^{-\sqrt{\bar\rho\nu t}}.
$$
A similar approach provides
$$
\norm{h}_{ A} \leq\norm{h_0}_{ A_\nu} e^{-\sqrt{\bar\rho\nu t}}.
$$
\end{proof}
\section{Local existence}\label{section:localexistence}
This section is devoted to prove local in time existence of solutions for the contour dynamics problem in the case of an interface defined by a curve $z(\alpha,t)$ \eqref{eq:curvaarbitraria}. We will prove that the local existence of smooth initial data is guaranteed regardless of whether the fluids are in a stable/unstable stratification.

As we need to control the arc-chord condition of the interface to avoid self-intersections, we define the function
\begin{equation}\label{arcchord}
	\mathcal{F}(z)(\alpha,\beta,t) = \frac{\beta^2}{\cosh (z_2(\alpha,t)-z_2(\alpha-\beta,t))-\cos (z_1(\alpha,t)-z_1(\alpha-\beta,t))} 
\end{equation}
for $\alpha, \beta \in (-\pi,\pi)$ and
\begin{equation}\label{arcchord0}
	\mathcal{F}(z)(\alpha,0,t) = \frac{2}{|\partial_\alpha z(\alpha,t)|^2}.
\end{equation}
Defining 
\begin{equation*}
	\norm{z(t)}_{L^\infty} = \sup_{\alpha \in \TT} \seminorm{z(\alpha,t) },
\end{equation*} 
and
\begin{equation*}
	\norm{\mathcal{F}(z)}_{L^\infty}  = \sup_{\alpha,\beta \in \TT} \seminorm{\mathcal{F}(z)},
\end{equation*} 
we emphasize that as long as 
$$
\norm{\mathcal{F}(z)}_{L^\infty} < \infty,
$$ self-intersections of the curve are excluded. 

Our result is the following
\begin{prop}[Local existence of solutions in  $C^{1,\gamma}(\TT)$]\label{thmlocal}
Let $0<\gamma<1$ be a fixed parameter and $z_0(\alpha) \in C^{1,\gamma}(\TT)$ be the initial data satisfying
$$
\norm{\mathcal{F}(z_0)}_{L^\infty}  < \infty.
$$ 
Then, there is a time $0<T$ such that there exists a unique solution 
$$
z \in C^1((-T,T);C^{1,\gamma}(\TT)),
$$ of \eqref{eq:curvaarbitraria} satisfying the arc-chord condition.
\end{prop}
\begin{proof}
\textbf{Sketch of the proof and functional framework;} The proof of this result is based on the Picard Theorem on a suitable Banach Space (see \cite{majda_vorticity_2001} for another application of this method in a different problem). The precise statement of Picard Theorem that we are going to use is
	\begin{lem}[Picard Theorem on a Banach Space]\label{picard}
		Let $O \subseteq B$ be an open subset of a Banach space $B$ and let $N(X)$ be a nonlinear operator satisfying the following criteria: 
		\begin{itemize}
			\item $N(X)$ maps $O$ to $B$,
			\item $N(X)$ is locally Lipschitz continuous, i.e., for any $X \in O$ there exists $L >0$ such that 
			\begin{equation*}
				\norm{N(\tilde{X})-N(X)}_B \leq L \norm{\tilde{X}-X}_B \quad \forall \, \tilde{X}  \in O.
			\end{equation*}
			Then, for any $X_0 \in O$, there exists a time $T>0$ such that the ODE 
			\begin{equation*}
				\frac{d X}{dt} = N(X), \quad X|_{t=0} = X_0 \in O
			\end{equation*}
			has a unique local solution $X \in C^1((-T,T);O)$.
		\end{itemize}
\end{lem} 
Our functional frame will be the Banach space $B=C^{1,\gamma}(\TT)$, equipped with the norm 
 \begin{equation*}
 \norm{z(t)}_{C^{1,\gamma}} = \norm{z(t)}_{L^\infty} + \norm{z_\alpha(t)}_{L^\infty}  + \seminorm{z_\alpha(t)}_\gamma,
 \end{equation*}
where the $\gamma-$Hölder seminorm is defined as
\begin{equation*}
\seminorm{z(t)}_\gamma = \max_{\alpha \neq \beta} \frac{|z(\alpha,t)-z(\beta,t)|}{|\alpha - \beta|^\gamma}.
\end{equation*}
We define the set $O^M \subset C^{1,\gamma}(\TT)$ as
 \begin{equation}\label{open}
 	O^M = \{ z \in C^{1,\gamma}(\TT); \,\, \norm{\mathcal{F}(z)}_{L^\infty}  < M, \, \, \norm{z}_{C^{1,\gamma}} < M   \}.
 \end{equation}
This set is open in $C^{1,\gamma}$ in the appropriate topology. Defining 
\begin{align} \label{b1}
N(z)(\alpha,t) =   (\rho^- - \rho^+)  \int_{\TT} \St(z(\alpha,t)-z(\beta,t)) z_\beta^\perp(\beta,t) z_2(\beta,t) d\beta, 
\end{align}
equation \eqref{eq:curvaarbitraria} can be written as
$$
\frac{d z(\alpha,t)}{dt} = N(z)(\alpha,t).
$$
Once we have established the functional setting of the theorem, it remains to prove the nonlinear operator satisfies the two hypothesis of the Picard Theorem. 

\textbf{Step 1: $N(z)$ maps $O^M$ to $C^{1,\gamma}(\TT)$;} We need to prove 
	\begin{equation*}
	\norm{N(z)}_{C^{1,\gamma}} \leq C(\rho,\gamma, M), \text{ for all } z \in O^M.
	\end{equation*}
	The proof of this proposition boils down to the following estimate:
	\begin{align}
\label{est2} &  \seminorm{N_\alpha(z)}_\gamma \leq C(\rho,\gamma,M).
	\end{align}
We write
 \begin{equation*}
 N(z)(\alpha)  = \frac{ (\rho^- - \rho^+) }{8 \pi } \left( N_1(z)(\alpha) - N_2(z)(\alpha) \right) 
 \end{equation*}
with
\begin{equation}\label{N1}
N_1(z)(\alpha) = \int_\TT \log \left(2(\cosh ({z_2}(\alpha)-z_2(\beta)) - \cos( {z_1}(\alpha)-z_1(\beta)))  \right) z_2(\beta) z_\beta^\perp (\beta) d \beta,
\end{equation}
\begin{align*}
&N_2(z)(\alpha) =   \int_\TT  \frac{({z_2}(\alpha)-z_2(\beta) ) z_2(\beta)  }{ \cosh( {z_2}(\alpha)-z_2(\beta) )- \cos ( {z_1}(\alpha)-z_1(\beta)) } \\
& \quad \times  \left( \begin{array}{cc}
- \sinh ({z_2}(\alpha)-z_2(\beta)) & \sin ( {z_1}(\alpha)-z_1(\beta))  \\
\sin ( {z_1}(\alpha)-z_1(\beta))  & 	\sinh({z_2}(\alpha)-z_2(\beta)) 
\end{array} \right) z_\alpha^\perp (\beta) d \beta  .
\end{align*}
For the sake of brevity, we are going to estimate only the term $N_1$, being $N_2$ similar.
To obtain an expression of $\partial_\alpha N_1(z)(\alpha)$, we differentiate the expression \eqref{N1} and we perform the change of variables $\beta = \alpha - \beta'$ for simplicity. We get
\begin{equation}\label{dN1}
\partial_\alpha N_1(z)(\alpha)= \int_\TT \frac{ z_\alpha(\alpha) \cdot  ( \sin ( {z_1}_-(\alpha) ),  \sinh ( {z_2}_-(\alpha) ))  }{ \beta^2}  \FF(z)(\alpha,\beta) z_2(\alpha-\beta) z_\alpha^\perp (\alpha-\beta)d \beta,
\end{equation}
where 
\begin{equation}\label{dif}
	{z_i}_-(\alpha) = z_i(\alpha)-z_i(\alpha-\beta).
\end{equation}

We have the following splitting
\begin{equation}\label{dN1h}
\partial _\alpha N_1(z) (\alpha +h) - \partial_\alpha N_1 (z) (\alpha) = I_1 + I_2 + I_3 + I_4+I_5,
\end{equation}
where
\begin{align*}
I_1 &= \int_\TT  \frac{( z_\alpha(\alpha+h)- z_\alpha(\alpha) )\cdot (\sin ( {z_1}_-(\alpha+h) ) ,\sinh ( {z_2}_-(\alpha+h)) ) }{\beta^2 } \\
& \quad \times   \FF(z)(\alpha+h,\beta)  z_2(\alpha+h-\beta) z_\alpha^\perp (\alpha+h-\beta)d \beta ,
\end{align*}
\begin{align*}
I_2 &= \int_\TT  \frac{z_\alpha(\alpha) \cdot  (\sin ( {z_1}_-(\alpha+h)) -\sin(  {z_1}_-(\alpha) ), \sinh ( {z_2}_-(\alpha+h)) -\sinh ( {z_2}_-(\alpha))  ) }{ \beta^2} \\
& \quad \times  \FF(z)(\alpha+h,\beta)   z_2(\alpha+h-\beta) z_\alpha^\perp (\alpha+h-\beta)d \beta ,
\end{align*}
\begin{align*}
I_3 &= \int_\TT z_\alpha(\alpha) \cdot  ( \sin(  {z_1}_-(\alpha) ),  \sinh(  {z_2}_-(\alpha)) ) \\
& \quad \times   \left(  \FF(z)(\alpha+h,\beta) - \FF(z)(\alpha,\beta)   \right)  \frac{z_2(\alpha+h-\beta) z_\alpha^\perp (\alpha+h-\beta)}{\beta^2}d \beta ,
\end{align*}
\begin{align*}
I_4 &= \int_\TT  \FF(z)(\alpha,\beta)   \frac{z_\alpha(\alpha) \cdot  ( \sin ( {z_1}_-(\alpha) ),  \sinh(  {z_2}_-(\alpha)) )  }{\beta^2  } (z_2(\alpha+h-\beta)-z_2(\alpha-\beta) )z_\alpha^\perp (\alpha-\beta)d \beta ,
\end{align*}
and
\begin{align*}
I_5 &= \int_\TT  \FF(z)(\alpha,\beta)  \frac{z_\alpha(\alpha) \cdot  ( \sin ( {z_1}_-(\alpha)) ,  \sinh ( {z_2}_-(\alpha)) ) )   }{ \beta^2} z_2(\alpha-\beta) (z_\alpha^\perp (\alpha+h-\beta)-z_\alpha^\perp (\alpha-\beta))d \beta .
\end{align*}
To prove $\seminorm{\partial_\alpha N_1(z)}_{\gamma}  \leq C(\gamma,M)$, it suffices to prove 
\begin{equation}\label{1to5}
|I_i|\leq |h|^\gamma C(\gamma,M) \text{ for } i = 1 \dots 5.
\end{equation} 
We are going to estimate only the first term, $I_1$, being the remainder terms similar. We compute
\begin{align*}
I_1 &= \int_\TT  \frac{( z_\alpha(\alpha+h)- z_\alpha(\alpha) )\cdot (\sin ( {z_1}_-(\alpha+h) ) ,\sinh ( {z_2}_-(\alpha+h)) ) }{\beta^2 } \\
& \quad \times \FF(z)(\alpha+h,\beta)  z_2(\alpha+h-\beta) z_\alpha^\perp (\alpha+h-\beta)d \beta \\
& = I_{1,1} + I_{1,2} + I_{1,3} + I_{1,4} + I_{1,5} + I_{1,6}.
\end{align*}
The first of these terms is
\begin{align*}
I_{1,1} &= \int_\TT  \frac{( z_\alpha(\alpha+h)- z_\alpha(\alpha) )\cdot \left[ (\sin( {z_1}_-(\alpha+h)),\sinh ( {z_2}_-(\alpha+h)))-  {z}_-(\alpha+h)   \right] }{ \beta^2 }  \\
&\quad  \times \FF(z)(\alpha+h,\beta)  z_2(\alpha+h-\beta) z_\alpha^\perp (\alpha+h-\beta)d \beta.
\end{align*}
Using the integral Mean Value Theorem 
\begin{equation}\label{meanvaluethm}
	z_-(\alpha+h) = \int_0^1  z'( (\alpha+h)s + (1-s)(\alpha+h-\beta) ) \beta ds,
\end{equation}
which, together with the series expansion of the trigonometric functions $\sin$ and $\sinh$, implies 
\begin{equation*}
	\seminorm{(\sin({z_1}_-(\alpha+h)),\sinh ({z_2}_-(\alpha+h)))-  {z}_-(\alpha+h)} \lesssim |\beta|^3 \norm{z_\alpha}_{L^\infty} ^3,
\end{equation*}
we find that
\begin{align*}
|I_{1,1}| &\lesssim |h|^\gamma \seminorm{z_\alpha}_\gamma  \norm{\mathcal{F}(z)}_{L^\infty} \norm{z_\alpha}_{L^\infty }^4 \norm{z}_{L^\infty} \int_\TT |\beta| d\beta, \\
& \leq |h|^\gamma C(\gamma,M).
\end{align*}

The second term reads
\begin{align*}
I_{1,2}  &= \int_\TT \left(   \FF(z)(\alpha+h,\beta)  - \frac{2}{ | z_\alpha(\alpha+h)|^2 } \right)  \frac{ ( z_\alpha(\alpha+h)- z_\alpha(\alpha) )\cdot {z}_-(\alpha+h) }{\beta^2} \\
& \quad \times  z_2(\alpha+h-\beta) z_\alpha^\perp (\alpha+h-\beta)d \beta.
\end{align*}
In order to estimate this term, we use \eqref{meanvaluethm} and 
\begin{align}\label{asympf}
	\seminorm{ \FF(z)(\alpha+h,\beta)  - \frac{2}{ | z_\alpha(\alpha+h)|^2 } } &= \seminorm{ \frac{\beta^2 | z_\alpha(\alpha+h)|^2- 2 (\cosh ({z_2}_-(\alpha+h)) - \cosh( {z_1}_-(\alpha+h) )     )   }{| z_\alpha(\alpha+h)|^2 (\cosh {z_2}_-(\alpha+h) - \cosh( {z_1}_-(\alpha+h)  )    )   } } \nonumber\\
	& \lesssim \beta^2 \norm{\FF(z)}_{L^\infty} ^2 \norm{z_\alpha}_{L^\infty}^4.
\end{align}
to find that
\begin{align*}
|I_{1,2}| &\lesssim |h|^\gamma \seminorm{z_\alpha}_\gamma \norm{\FF(z)}_{L^\infty} ^2 \norm{z}_{L^\infty}^2 \norm{z_\alpha}_{L^\infty}^5, \\
& \leq |h|^\gamma C(\gamma,M).
\end{align*}
Note that thanks to \eqref{arcchord0}, it holds that 
\begin{equation*}
	\frac{2}{\seminorm{z_\alpha(\cdot )}^2} = \FF(\cdot,\beta,t) \leq M.
\end{equation*}
Then, the third term
\begin{align*}
I_{1,3} &= \frac{2}{ | z_\alpha(\alpha+h)|^2}  \int_\TT  \frac{  ( z_\alpha(\alpha+h)- z_\alpha(\alpha) )\cdot  {z}_-(\alpha+h) }{\beta} \frac{z_2(\alpha+h-\beta)- z_2(\alpha+h) }{\beta}z_\alpha^\perp (\alpha+h-\beta)d \beta ,
\end{align*}
can be bounded as
\begin{align*}
|I_{1,3}| &\lesssim |h|^\gamma \seminorm{z_\alpha}_\gamma  \norm{\mathcal{F}(z)}_{L^\infty} \norm{z_\alpha}_{L^\infty}^3,  \\
& \leq |h|^\gamma C(\gamma,M) .
\end{align*}
Using the previous ideas together with the inequality
\begin{equation}\label{gammaest}
	\seminorm{z(\alpha+h)-z(\alpha+h-\beta) - z_\alpha(\alpha+h) \beta} \leq |\beta|^{1+\gamma} |z_\alpha|_\gamma ,
\end{equation}
we have that
\begin{align*}
I_{1,4}  &= \frac{2z_2(\alpha+h)}{ | z_\alpha(\alpha+h)|^2}  \int_\TT  \frac{   ( z_\alpha(\alpha+h)- z_\alpha(\alpha) )\cdot ({z}_-(\alpha+h) -\beta z_\alpha(\alpha+h) ) }{\beta^2}z_\alpha^\perp (\alpha+h-\beta)d \beta ,
\end{align*}
verifies
\begin{align*}
|I_{1,4} | &\lesssim  |h|^\gamma \seminorm{z_\alpha}_\gamma^2 \norm{\mathcal{F}(z)}_{L^\infty} \norm{z}_{L^\infty} \norm{z_\alpha}_{L^\infty}  \int_\TT |\beta|^{1+\gamma-2} d\beta, \\
& \lesssim  |h|^\gamma \seminorm{z_\alpha}_\gamma^2  \norm{\mathcal{F}(z)}_{L^\infty} \norm{z}_{L^\infty} \norm{z_\alpha}_{L^\infty} , \\
& \leq |h|^\gamma C(\gamma,M) .
\end{align*}
The term
\begin{align*}
I_{1,5}  &= \frac{2z_2(\alpha+h) (z_\alpha(\alpha+h)-z_\alpha(\alpha)) \cdot z_\alpha(\alpha+h)}{ | z_\alpha(\alpha+h)|^2}  \int_\TT \left(   \frac{1}{\beta} - \cot (\beta) \right)z_\alpha^\perp (\alpha+h-\beta)d \beta ,
\end{align*}
can be estimated as follows
\begin{align*}
|I_{1,5}| &\lesssim |h|^\gamma \seminorm{z_\alpha}_\gamma   \norm{\mathcal{F}(z)}_{L^\infty} \norm{z}_{L^\infty} \norm{z_\alpha}_{L^\infty}^2 \int_\TT |\beta|d\beta, \\
& \lesssim |h|^\gamma \seminorm{z_\alpha}_\gamma   \norm{\mathcal{F}(z)}_{L^\infty} \norm{z}_{L^\infty} \norm{z_\alpha}_{L^\infty}^2, \\
& \leq |h|^\gamma C(\gamma,M) .
\end{align*}
Finally, the last term reads
\begin{align*}
I_{1,6} &= \frac{2z_2(\alpha+h)  (z_\alpha(\alpha+h)-z_\alpha(\alpha)) \cdot z_\alpha(\alpha+h)}{ | z_\alpha(\alpha+h)|^2}  2 \pi \mathcal{H}(z_\alpha^\perp)(\alpha+h).
\end{align*}
We observe that the Hilbert transform $\mathcal{H}(\cdot)$ given in \eqref{hilbert} maps continuously $C^\gamma$ to $C^\gamma$. In particular, 
\begin{equation}\label{hilbertgamma}
	\norm{\mathcal{H}(z_\alpha^\perp)}_{L^\infty } \lesssim \norm{z_\alpha^\perp}_{\gamma}.
\end{equation}
Consequently,
\begin{align*}
|I_{1,6}| &\lesssim  |h|^\gamma \seminorm{z_\alpha}_\gamma \norm{\mathcal{F}(z)}_{L^\infty} \norm{z}_{L^\infty} \norm{z_\alpha}_{L^\infty}  \norm{z_\alpha}_{\gamma} , \\
& \leq |h|^\gamma C(\gamma,M) .
\end{align*}
Collecting the estimates of the terms $I_{1,1} \dots I_{1,6}$, we obtain
\begin{equation}\label{I1est}
	|I_{1}| \leq |h|^\gamma C(\gamma,M) .
	\end{equation}
Using these ideas we can estimate the remaining terms. This concludes the proof of this step.

\textbf{Step 2: $N(z)$ is locally Lipschitz continuous;}
We have to prove that the operator $N$ defined as in \eqref{b1} is locally Lipschitz continuous in $C^{1,\gamma}(\TT)$, i.e., for any $z \in O^M$, there exists a constant $L>0$ such that 
 	\begin{equation}\label{lip}
 	\norm{N(z)-N(w))}_{C^{1,\gamma}} \leq L(\rho,\gamma,M) \norm{z-w}_{C^{1,\gamma}} \quad \forall \, w \in O^M.
 	\end{equation}
For the sake of brevity, we are going to show some of the estimates corresponding to 
	\begin{equation*}
\seminorm{N_\alpha(z)-N_\alpha(w))}_{\gamma}\leq L(\rho,\gamma,M) \norm{z-w}_{C^{1,\gamma}}.
\end{equation*}
	The previous statement is equivalent to prove that 
		\begin{equation}\label{lipest}
	\seminorm{(N_\alpha(z)(\alpha+h)-N_\alpha(w)(\alpha+h))-(N_\alpha(z)(\alpha)-N_\alpha(w)(\alpha))}\leq L(\rho,\gamma,M)  |h|^\gamma \norm{z-w}_{C^{1,\gamma}} .
	\end{equation}
As before, we are going to focus on the term $N_1$. From \eqref{dN1h}, we deduce that 
\begin{equation*}
	(\partial_\alpha N_1(z)(\alpha+h)-\partial_\alpha N_1(w)(\alpha+h))-(\partial_\alpha N_1(z)(\alpha)-\partial_\alpha N_1(w)(\alpha)) = \sum_{i=1}^5 (I_i(z)-I_i(w)).
\end{equation*}
Let us focus on the first term. The remaining terms can be estimated using the same tools and ideas. The first term can be split as follows:
\begin{equation*}
	I_1(z)-I_1(w) = J_1^1 + J_2^1 + J_3^1 + J_4^1 + J_5^1,
\end{equation*}
where 
\begin{align*}
	J_1^1 &= \int_\TT  \frac{( (z-w)_\alpha(\alpha+h)- (z-w)_\alpha(\alpha) )\cdot (\sin ( {z_1}_-(\alpha+h))  ,\sinh ( {z_2}_-(\alpha+h)) ) }{\beta^2 } \\
	& \quad\quad \times   \FF(z)(\alpha+h,\beta)  z_2(\alpha+h-\beta) z_\alpha^\perp (\alpha+h-\beta)d \beta,
\end{align*}
\begin{align*}	
	J_2^1 &= \int_\TT  \frac{ (\sin ( {z_1}_-(\alpha+h))-\sin ( {w_1}_-(\alpha+h))  ,\sinh ( {z_2}_-(\alpha+h))-\sinh ( {w_2}_-(\alpha+h)) ) }{\beta^2 }\\
	& \quad\quad \cdot ( w_\alpha(\alpha+h)- w_\alpha(\alpha) )   \FF(z)(\alpha+h,\beta)  z_2(\alpha+h-\beta) z_\alpha^\perp (\alpha+h-\beta)d \beta ,
\end{align*}
\begin{align*}
	J_3^1 &= \int_\TT  \frac{( w_\alpha(\alpha+h)- w_\alpha(\alpha) )\cdot (\sin ( {w_1}_-(\alpha+h))  ,\sinh(  {w_2}_-(\alpha+h)) ) }{\beta^2 } \\
	& \quad\quad \times   (\FF(z)(\alpha+h,\beta) -\FF(w)(\alpha+h,\beta)) z_2(\alpha+h-\beta) z_\alpha^\perp (\alpha+h-\beta)d \beta ,
\end{align*}
\begin{align*}
	J_4^1 &= \int_\TT  \frac{( w_\alpha(\alpha+h)- w_\alpha(\alpha) )\cdot (\sin ( {w_1}_-(\alpha+h) ) ,\sinh ( {w_2}_-(\alpha+h)) ) }{\beta^2 } \\
	& \quad\quad \times   \FF(w)(\alpha+h,\beta)  (z_2-w_2)(\alpha+h-\beta) z_\alpha^\perp (\alpha+h-\beta)d \beta ,
\end{align*}
and
\begin{align*}
	J_5^1 &= \int_\TT  \frac{( w_\alpha(\alpha+h)- w_\alpha(\alpha) )\cdot (\sin ( {w_1}_-(\alpha+h) ) ,\sinh ( {w_2}_-(\alpha+h)) ) }{\beta^2 }  \\
	&\quad\quad \times \FF(w)(\alpha+h,\beta)  w_2(\alpha+h-\beta) (z-w)_\alpha^\perp (\alpha+h-\beta)d \beta .
\end{align*}
The terms $J_1^1, J_4^1$ and $J_5^1$ can be estimated similarly to $I_1$ in Step 1. As a consequence we obtain the estimates 
\begin{equation*}
	|J_i^1| \leq L(\gamma,M) \norm{z-w}_{C^{1,\gamma}} \text{ for } i = 1,4,5.
\end{equation*}
We treat the terms $J_2^1, J_3^1$ separately, as the difference of the curves is not explicitly present in these terms. We split $J_2^1$ as follows: 

\begin{align*}
	J_2^1 &= \int_\TT  \frac{ (\sin  ({z_1}_-(\alpha+h))-\sin ( {w_1}_-(\alpha+h))  ,\sinh (  {z_2}_-(\alpha+h))-\sinh ( {w_2}_-(\alpha+h)) ) }{\beta^2 } \\
	& \quad \cdot ( w_\alpha(\alpha+h)- w_\alpha(\alpha) )  \FF(z)(\alpha+h,\beta)  z_2(\alpha+h-\beta) z_\alpha^\perp (\alpha+h-\beta)d \beta  \\
	& =  J_{2,1}^1 + J_{2,2}^1 + J_{2,3}^1 + J_{2,4}^1 + J_{2,5}^1 + J_{2,6}^1.
\end{align*}
We have that
\begin{align*}
		J_{2,1}^1 &= \\  \int_\TT &  \left[ (\sin ( {z_1}_-(\alpha+h) )-\sin ( {w_1}_-(\alpha+h))  ,\sinh ( {z_2}_-(\alpha+h))-\sinh ( {w_2}_-(\alpha+h)) ) - (z-w)_-(\alpha+h)\right] \\
	&\quad  \cdot( w_\alpha(\alpha+h)- w_\alpha(\alpha) )   \frac{\FF(z)(\alpha+h,\beta)  z_2(\alpha+h-\beta) z_\alpha^\perp (\alpha+h-\beta)}{\beta^2} d\beta
\end{align*}
\begin{align*}
	J_{2,2}^1 = \int_\TT    (z-w)_-(\alpha+h)&  \cdot  ( w_\alpha(\alpha+h)- w_\alpha(\alpha) )    \left( \FF(z)(\alpha+h,\beta) - \frac{2}{|z_\alpha(\alpha+h)|^2}   \right) \\
	&\times \frac{ z_2(\alpha+h-\beta) z_\alpha^\perp (\alpha+h-\beta)}{\beta^2}  d\beta ,
\end{align*}
\begin{align*}
	J_{2,3}^1 =  \frac{2}{|z_\alpha(\alpha+h)|^2} \int_\TT    (z-w)_-(\alpha+h)&  \cdot  ( w_\alpha(\alpha+h)- w_\alpha(\alpha) ) \\
	&\times   \frac{ (z_2(\alpha+h-\beta)-z_2(\alpha+h)) z_\alpha^\perp (\alpha+h-\beta)}{\beta^2}  d\beta ,
\end{align*}
\begin{align*}
	J_{2,4}^1 =  \frac{2z_2(\alpha+h)}{|z_\alpha(\alpha+h)|^2} \int_\TT   ( (z-w)_-(\alpha+h)& -\beta(z-w)_\alpha(\alpha+h)) \cdot  ( w_\alpha(\alpha+h)- w_\alpha(\alpha) )  \\
	&\times  \frac{  z_\alpha^\perp (\alpha+h-\beta)}{\beta^2}  d\beta ,
\end{align*}
\begin{align*}
	J_{2,5}^1 &=  \frac{2z_2(\alpha+h)}{|z_\alpha(\alpha+h)|^2} (z-w)_\alpha(\alpha+h) \cdot  ( w_\alpha(\alpha+h)- w_\alpha(\alpha) )  \int_\TT z_\alpha^\perp (\alpha+h-\beta)  \left(  \frac{  1}{\beta}-\cot(\beta) \right)  d\beta ,
\end{align*}
and
\begin{align*}
	J_{2,6}^1 &=  \frac{2z_2(\alpha+h)}{|z_\alpha(\alpha+h)|^2} (z-w)_\alpha(\alpha+h) \cdot  ( w_\alpha(\alpha+h)- w_\alpha(\alpha) )  2 \pi \mathcal{H}(z_\alpha^\perp) (\alpha+h)  .
\end{align*}
To estimate the term $J_{2,1}^1$, we use the integral Mean Value Theorem and the Taylor series expansion of the functions $\sin$ and $\sinh$ leading us to
\begin{align*}
&\seminorm{	(\sin ( {z_1}_-(\alpha+h))-\sin ( {w_1}_-(\alpha+h) ) ,\sinh ( {z_2}_-(\alpha+h))-\sinh ( {w_2}_-(\alpha+h)) ) - (z-w)_-(\alpha+h) } \\
&  \leq C(M)  |\beta|^3 \norm{ (z-w)_\alpha}_{L^\infty} .
\end{align*}
which implies 
\begin{align*}
	\seminorm{J_{2,1}^1} & \leq C(M) \norm{(z-w)_\alpha}_{L^\infty}  |h|^\gamma \seminorm{w_\alpha}_\gamma \norm{F(z)}_{L^\infty}  \norm{z}_{L^\infty}  \norm{z_\alpha}_{L^\infty}  \int_\TT |\beta| d\beta, \\
	& \leq L(\gamma.M) |h|^\gamma \norm{z-w}_{C^{1,\gamma}}.
\end{align*}
The terms $J_{2,2}^1 \dots J_{2,6}^1$ are estimated like the terms $I_{1,2} \dots I_{1,6}$. Note that they contain terms depending on the linear difference of the curves $z-w$, so consequently we obtain estimates of the form: 
\begin{equation*}
	|J_{2,i}^1| \leq L(\gamma,M) |h|^\gamma \norm{z-w}_{C^{1,\gamma}} \text{ for } i = 2 \dots 6.
\end{equation*}
We conclude that 
\begin{equation*}
		|J_{2}^1| \leq L(\gamma,M) |h|^\gamma \norm{z-w}_{C^{1,\gamma}} .
\end{equation*}

In order to estimate the term $J_3^1$ we have to use
\begin{align}
	&  \FF(z)(\cdot,\beta) -\FF(w)(\cdot,\beta)\nonumber \\
 &\quad	= \FF(z)(\cdot,\beta) \FF(w)(\cdot,\beta) \frac{\cosh( {w_2}_-(\cdot) )-\cos ({w_1}_-(\cdot) ) - \cosh( {z_2}_-(\cdot)) +\cos ({z_1}_-(\cdot)) }{\beta^2} ,  \label{zwdifarcchord}
\end{align}
which, together with
\begin{align*}
	&\seminorm{\cosh( {w_2}_-(\alpha+h)) -\cos( {w_1}_-(\alpha+h) ) - \cosh ({z_2}_-(\alpha+h)) +\cos( {z_1}_-(\alpha+h))} \\ &\leq C(M) |\beta|^2 \norm{(z-w)_\alpha}_{L^\infty}  ,
\end{align*}
lead to
\begin{equation*}
\seminorm{	 \FF(z)(\alpha+h,\beta) -\FF(w)(\alpha+h,\beta)} \leq C(M) \norm{(z-w)_\alpha}_{L^\infty}  .
\end{equation*}
Similarly, we find that
\begin{align*}
	   \left\vert \FF(z)(\alpha+h,\beta) -\frac{2}{\seminorm{z_\alpha(\alpha+h)}^2}-\FF(w)(\alpha+h,\beta) + \frac{2}{\seminorm{w_\alpha(\alpha+h)}^2} \right\vert \leq C(M) |\beta|^2 \norm{(z-w)_\alpha}_{L^\infty }.
\end{align*}
This leads to 
\begin{equation*}
	\seminorm{J_{3}^1} \leq L(\gamma,M) |h|^\gamma \norm{z-w}_{C^{1,\gamma}}.
\end{equation*}
As a consequence, we have that 
\begin{equation*}
|	I_1(z)-I_1(w)| \leq L(\gamma,M) |h|^\gamma \norm{z-w}_{C^{1,\gamma}}.
\end{equation*}

We can perform the same type of splitting for the other terms $I_i(z)-I_i(w)$. Then, the desired estimate is obtained using the same ideas. This concludes the proof of the result.
\end{proof}

\section{Global existence and decay of Sobolev solutions in the RT stable case}\label{section:globalexistence}
In this section we assume that the system is in the Rayleigh-Taylor stable case where the lighter fluid lies above the denser fluid, namely
\begin{equation*}
\rho^- - \rho^+ >0.
\end{equation*}
Then we prove the global existence and decay to equilibrium for small enough initial data  \begin{equation*}
	z(\alpha,0) = (\alpha,h(\alpha,0))
\end{equation*}
given as a graph in the Sobolev space $H^3$. In particular, the result we prove is the following 
\begin{theo}[Global existence and decay of solutions for small data] \label{stabilitythm}
Let $h_0\in H^3(\TT)$ be the initial data for \eqref{eq:grafo2} in the RT stable case,
\begin{equation*}
\rho^- - \rho^+>0,
\end{equation*}
and take $\frac{3}{2}<m<2$. There is a $0<\delta= \delta(\rho^- - \rho^+)$ such that if
$$
\norm{h_0}_{   H^{3}(\TT)}<\delta
$$ 
there exists a unique global classical solution $h(\alpha,t)$ 
$$
h\in C([0,T];H^3)
$$
of \eqref{eq:grafo2} for arbitrary $T>0$ satisfying
$$
(1+t)^{m} \norm{h}_{L^2}+ \norm{h}_{{  H}^3}\leq C \norm{h_0}_{H^3} 	
$$	
for some constant $C >0$.
\end{theo}
\begin{proof}
\textbf{The linear part:}
 First we observe that, without lost of generality, we can consider the initial data $h_0(\alpha)$ having zero mean. Due to the fact that the equation for the case of a graph can be written as
\begin{equation*}
	\partial_t h(\alpha,t) d \alpha =  \partial_\alpha \psi(\alpha,h(\alpha,t)) 
\end{equation*}
where $\psi$ is the stream function, we have that the solution $h(\alpha,t)$ satisfies
\begin{equation*}
	\int_{\TT} h(\alpha,t) d \alpha = 0  \text{ for every } t >0.
\end{equation*} 
In this setting, equation \eqref{eq:grafo2} can be written as
\begin{equation}\label{isolatelinear}
	h_t(\alpha,t) =-\bar\rho \Lambda^{-1}(h)(\alpha,t) + \frac{\bar\rho}{2 \pi}(I_1(\alpha,t)+I_2(\alpha,t)+I_3(\alpha,t)+I_4(\alpha,t)) ,
\end{equation}
where
$$
\bar{\rho}=\rho^- - \rho^+,
$$
\begin{equation*}
	I_1(\alpha,t) =   \int_\TT \log \left(  1+\frac{\sinh^2\left(\frac{h_-(\alpha,t)}{2}\right)  }{\sin^2\left(\frac{\beta}{2}\right) }  \right) h(\alpha-\beta,t) d \beta ,
\end{equation*}
\begin{equation*}
	I_2(\alpha,t) = \int_\TT \log \left( \sinh^2\left(\frac{h_-(\alpha,t)}{2}\right) + \sin^2\left(\frac{\beta}{2}\right)     \right) h(\alpha-\beta,t) h_\alpha(\alpha,t) h_\alpha (\alpha-\beta,t) d\beta,
\end{equation*}
\begin{equation*}
	I_3(\alpha,t) =   \int_\TT  \frac{h(\alpha-\beta,t) h_-(\alpha,t) }{2\sinh^2\left(\frac{h_-(\alpha,t)}{2}\right) + 2\sin^2\left(\frac{\beta}{2}\right)  } \left[ (h_\alpha(\alpha,t)h_\alpha(\alpha-\beta,t)-1) \sinh(h_-(\alpha,t))\right] d\beta,
\end{equation*}
\begin{equation*}
	I_4(\alpha,t) =  \int_\TT \frac{h(\alpha-\beta,t) h_-(\alpha,t)}{2\sinh^2\left(\frac{h_-(\alpha,t)}{2}\right) +2 \sin^2\left(\frac{\beta}{2}\right)  } \left[(h_\alpha (\alpha,t) + h_\alpha (\alpha-\beta,t) ) \sin (\beta)   \right] d\beta.  
\end{equation*}
Above we use the notation $h_-(\alpha)= h(\alpha)-h(\alpha-\beta)$ as in previous sections. Define 
\begin{equation}\label{tnorm}
	\tnorm{h} = \sup_{t \in [0,T]} \left((1+t)^{m} \norm{h}_{L^2}+ \norm{h}_{{  H}^3}\right). 	
\end{equation}
The proof follows the ideas in \cite{cheng_well-posedness_2016}. Namely, our goal now is to obtain an inequality of the form
$$
\tnorm{h}\leq \mathcal{M}_0+F(\tnorm{h}),
$$
where $F$ is a smooth $O(x^k)$ function with $k>1$.

\textbf{The $L^2$ estimate:} In this subsection we prove a $L^2$ estimate of the interface $h(\alpha,t)$ by using Duhamel principle. Using Duhamel formula we obtain that
\begin{align*}
h(\alpha,t) = e^{- \bar\rho \Lambda^{-1} t} h_0(\alpha) + \frac{\bar\rho}{2 \pi}\int_0^t e^{- \bar\rho \Lambda^{-1} (t-s)} (I_1(\alpha,s)+I_2(\alpha,s)+I_3(\alpha,s)+I_4(\alpha,s)) ds.
\end{align*}
Then, by \eqref{lineardecay},
\begin{equation}\label{duhamel}
\norm{h}_{L^2} \lesssim (1+t)^{-m}\norm{h_0}_{ H^{2}} + \frac{\bar\rho}{2 \pi}\int_0^t (1+t-s)^{-m} \norm{ (I_1(s)+I_2(s)+I_3(s)+I_4(s)) }_{ H^{2}}ds
\end{equation}
with $m<2$ to be chosen later.

We first estimate each of the terms $I_1(\alpha,s)\dots I_4(\alpha,s)$ in ${ H}^2$.  Secondly, we use Gagliardo-Niremberg interpolation inequality, in order to get estimates in our norm, defined by \eqref{tnorm}. Some  basic inequalities we will use are: 
\begin{align*}
	\norm{\partial_\alpha^kh(t)}_{L^\infty} \lesssim \norm{h(t)}_{{ H}^3}^{(1+2k)/6} \norm{h(t)}_{L^2}^{(5-2k)/6} \lesssim (1+t)^{-m \frac{(5-2k)}{6}} \tnorm{h},\;0\leq k\leq 2,\\
	\norm{h(t)}_{{ H}^k} \lesssim \norm{h(t)}_{{ H}^3}^{k/3} \norm{h(t)}_{L^2}^{(3-k)/3} \lesssim (1+t)^{-m \frac{(3-k)}{3}} \tnorm{h},\;0\leq k\leq 3.			
\end{align*}
For the sake of simplicity, we will drop the time dependence in the notation when it does not cause any ambiguity. We start by estimating $I_1$ in ${{ H}^2}$. We have that
\begin{align*}
	\norm{I_1}_{{ H}^2} & \lesssim K_1^1 + K_2^1 + K_3^1,
\end{align*}
where
\begin{align*}
K_1^1 &= 	\norm{ \int_\TT \log \left(  1+\frac{\sinh^2\left(\frac{h_-(\alpha)}{2}\right)  }{\sin^2\left(\frac{\beta}{2}\right) }  \right) \partial_\alpha^2 h(\alpha-\beta) d \beta}_{L^2} ,\\
K_2^1 &=  \norm{  \int_\TT \partial_\alpha \log \left(  1+\frac{\sinh^2\left(\frac{h_-(\alpha)}{2}\right)  }{\sin^2\left(\frac{\beta}{2}\right) }  \right)h_\alpha(\alpha-\beta) d \beta}_{L^2},
\end{align*}
and
\begin{align*}
K_3^1= \norm{\int_\TT \partial_\alpha^2 \log \left(  1+\frac{\sinh^2\left(\frac{h_-(\alpha)}{2}\right)  }{\sin^2\left(\frac{\beta}{2}\right) }  \right)  h(\alpha-\beta) d \beta}_{L^2}.
\end{align*}

First of all,
\begin{equation*}
	 K_1^1 \lesssim \sup_{\alpha, \beta \in \TT} \seminorm{ \log \left(  1+\frac{\sinh^2\left(\frac{h_-(\alpha)}{2}\right)  }{\sin^2\left(\frac{\beta}{2}\right) }  \right) }  \norm{ h}_{{{ H}^2}}  .
\end{equation*}
Using the integral Mean Value Theorem in \eqref{meanvaluethm}, we can estimate the finite differences in $h$ as 
\begin{equation*}
\seminorm{ \partial_\alpha^k  \frac{h_-(\alpha)}{\beta}  } \lesssim \norm{\partial_\alpha^{k+1} h}_{L^\infty}.
\end{equation*}
This fact together with 
 \begin{equation*}
	\frac{\sinh(x)}{x} \leq \cosh(x), \quad \frac{\beta}{2 \sin \left(\frac{\beta}{2}\right)} \leq \frac{\pi}{2},
\end{equation*}
 imply
\begin{align*}
	\seminorm{\log \left(  1+\frac{\sinh^2\left(\frac{h_-(\alpha)}{2}\right)  }{\sin^2\left(\frac{\beta}{2}\right) }  \right)} & \leq  \seminorm{\frac{\sinh\left(\frac{h_-(\alpha)}{2}\right)  }{\sin\left(\frac{\beta}{2}  \right)}}^2 ,\\
	& \leq  \seminorm{\frac{\sinh\left(\frac{h_-(\alpha)}{2}\right)  }{\frac{h_-(\alpha)}{2}}}^2  \seminorm{\frac{h_-(\alpha) }{\beta}}^2 \seminorm{\frac{ \frac{\beta}{2}}{  \sin \left(\frac{\beta}{2}\right)  }}^2   \\
	& \lesssim \cosh^2(\norm{h}_{L^\infty}) \norm{h_\alpha}_{L^\infty}^2.
\end{align*}

The previous computations imply
\begin{equation}\label{K11}
	K_1^1 \lesssim \cosh^2(\norm{h}_{L^\infty}) \norm{h_\alpha}_{L^\infty}^2  \norm{ h}_{{{ H}^2}}. 
\end{equation}
Similarly,
\begin{equation}\label{K21}
	K_2^1 \lesssim \cosh(2\norm{h}_{L^\infty}) \norm{ h_\alpha}_{L^\infty}^2 \norm{h}_{{ H}^2}
\end{equation}
and
\begin{align}\label{K31}
K_3^1 & \lesssim  \cosh^2(2\norm{h}_{L^\infty} ) \norm{h}_{L^\infty} \left(  \norm{h}_{{ H}^3}  \norm{h_\alpha}_{L^\infty} + \norm{h}_{{ H}^2} \norm{\partial_\alpha^2 h}_{L^\infty}+   \norm{h}_{{ H}^2} \norm{\partial_\alpha^2 h}_{L^\infty} \norm{h_\alpha}_{L^\infty}^2  \right).
\end{align}

Now, we apply Gagliardo-Nirenberg interpolation inequalities to \eqref{K11}, \eqref{K21} and \eqref{K31}, which will produce powers of the energy multiplied by factors of the type $(1+s)^{-m\mu}$ for some $\mu >0$. As a consequence, we find the estimate
\begin{equation}\label{I1H2}
	\norm{I_1}_{{ H}^2}  \lesssim  (1+s)^{-m\frac{4}{3} }   \cosh^2(2\norm{h}_{L^\infty} ) \left( \tnorm{h}^3\ + \tnorm{h}^5 \right).
\end{equation}

Now, let us estimate the second term $I_2$ using the same techniques:
\begin{equation}\label{I2H2}
	\norm{I_2}_{{ H}^2}  \lesssim (1+s)^{-m \frac{4}{3}} \left( 1+\cosh^2(2\norm{h}_{L^\infty} )     \right) \left( \tnorm{h}^3 + \tnorm{h}^5 + \tnorm{h}^7\right).
\end{equation}
The term $I_3$ is similarly estimated. In this way we find that
\begin{equation}\label{I3H2}
	\norm{I_3 }_{{ H}^2} \lesssim  (1+s)^{-m \frac{4}{3}} \cosh^3(4\norm{h}_{L^\infty}) \left( \tnorm{h}^3 +\tnorm{h}^5 + \tnorm{h}^7+ \tnorm{h}^9\right) .
\end{equation}
Finally, 
\begin{equation}\label{I4H2}
	\norm{I_{4}}_{{ H}^2} \lesssim (1+s)^{-m \frac{4}{3}} \left(1+ \cosh^2(2 \norm{h}_{L^\infty})  \right) \left( \tnorm{h}^3 + \tnorm{h}^5 + \tnorm{h}^7 \right).
\end{equation}

In order to close the estimate in $L^2$, we need integrability in time in \eqref{duhamel}. Note that the norm in $B$ does not depend on time. For that purpose, we use Lemma 2.4 in \cite{elgindi_asymptotic_2017}, which establishes that, for $m, \eta >0$,
\begin{equation}\label{elgindi}
	\int_0^t (1+t-s)^{-m} (1+s)^{-1-\eta} \leq C(m,\eta) (1+t)^{-m}.
\end{equation}
Introducing estimates \eqref{I1H2}, \eqref{I2H2}, \eqref{I3H2} and \eqref{I4H2} in \eqref{duhamel} and using \eqref{elgindi}, we obtain that, for any $m>\frac{3}{4} $, it holds that
\begin{equation}\label{L2estimate}
	\norm{h}_{L^2} \leq (1+t)^{-m}\norm{h_0}_{ H^{2}} + C(1+t)^{-m} \mathcal{B}_1(\tnorm{h}) \mathcal{P}_1(\tnorm{h}),
\end{equation}
where $\mathcal P_1$ is a polynomial with monomials at least cubic and 
\begin{equation*}
	\mathcal{B}_1(x) = 1 + \cosh^3(4x).
\end{equation*}

\textbf{The $H^3$ estimate:}
We compute that
\begin{align*}
\frac{1}{2}\frac{d}{dt} \norm{h(t)}_{ H^3}^2 & =  -\bar{\rho} \norm{h(t)}^2_{H^{2.5}} + \frac{\bar{\rho}}{2\pi}  \int_{\TT} \partial_\alpha^3 h(\alpha,t) \partial_\alpha^3\left( I_1(\alpha,t) + I_2(\alpha,t) + I_3(\alpha,t) + I_4(\alpha,t) \right) d \alpha.
\end{align*}
Then, 
\begin{align*}
	\frac{d}{dt} \norm{h}_{ H^3}^2 & \lesssim  \frac{\bar{\rho}}{2\pi}  \int_{\TT} \partial_\alpha^3 h(\alpha) \partial_\alpha^3\left( I_1(\alpha) + I_2(\alpha) + I_3(\alpha) + I_4(\alpha) \right) d \alpha.
\end{align*}
The aim is to estimate each of the four terms in $ H^3$. The strategy to bound most of the terms is similar to the one in \cite{cordoba_contour_2007}.

For the sake of brevity we will estimate in detail the term corresponding to the first integral. Denote 
\begin{equation*}
	 \int_\TT \partial_\alpha^3 h(\alpha) \partial_\alpha^3 I_1(\alpha) d \alpha = R_1^1 + R_2^1 + R_3^1 + R_4^1,
\end{equation*}
where 
\begin{align*}
	R_1^1 =   \int_\TT \partial_\alpha^3 h(\alpha) \left( \int_\TT \log \left(  1+\frac{\sinh^2\left(\frac{h_-(\alpha)}{2}\right)  }{\sin^2\left(\frac{\beta}{2}\right) }  \right)  \partial_\alpha^3 h(\alpha-\beta) d \beta \right) d \alpha
\end{align*}
\begin{align*}
	R_2^1 = 3   \int_\TT \partial_\alpha^3 h(\alpha)  \left( \int_\TT \partial_\alpha  \log \left(  1+\frac{\sinh^2\left(\frac{h_-(\alpha)}{2}\right)  }{\sin^2\left(\frac{\beta}{2}\right) }  \right)  \partial_\alpha^2 h(\alpha-\beta) d \beta\right) d \alpha,
\end{align*}
\begin{align*}
R_3^1=  3   \int_\TT \partial_\alpha^3 h(\alpha) \left(\int_\TT \partial_\alpha^2 \log \left(  1+\frac{\sinh^2\left(\frac{h_-(\alpha)}{2}\right)  }{\sin^2\left(\frac{\beta}{2}\right) }  \right)  \partial_\alpha h(\alpha-\beta) d \beta \right)d \alpha,
\end{align*}
and
\begin{align*}
R_4^1 =   \int_\TT \partial_\alpha^3 h(\alpha)\left(   \int_\TT \partial_\alpha^3 \log \left(  1+\frac{\sinh^2\left(\frac{h_-(\alpha)}{2}\right)  }{\sin^2\left(\frac{\beta}{2}\right) }  \right)   h(\alpha-\beta) d \beta\right) d \alpha .
\end{align*}

Similarly to the estimates performed for the $L^2$ estimate of the interface, we find that 
\begin{align*}
	R_1^1 & \lesssim \norm{h}_{{ H}^3}^2 \cosh^2(\norm{h}_{L^\infty}) \norm{h_\alpha}^2_{L^\infty} \\
	& \lesssim (1+t)^{-m} \cosh^2(\tnorm{h}) \tnorm{h}^4,
\end{align*}
\begin{align*}
	R_2^1 & \lesssim \norm{h}_{{ H}^3} \norm{h}_{{ H}^2} \cosh(\norm{h}_{L^\infty}) \norm{\partial_\alpha^2 h}_{L^\infty} \norm{h_\alpha}_{L^\infty} \\
	& \lesssim (1+t)^{-m} \cosh(\tnorm{h}) \tnorm{h}^4,
\end{align*}
\begin{align*}
	R_3^1 
	& \lesssim (1+t)^{-m} \cosh^2(\tnorm{h}) \left( \tnorm{h}^4 + \tnorm{h}^6 \right).
\end{align*}

The term $R_4^1$ is more singular. We split it into parts by computing the derivative   $$\partial_\alpha^3 \log \left(1+ \frac{\sinh^2(h_-(\alpha)/2)}{\sin^2(\beta/2)} \right)$$ explicitly:
\begin{align*}
	R_4^1 = R_{4,1}^1 + R_{4,2}^1 + R_{4,3}^1 + R_{4,4}^1 + R_{4,5}^1 + R_{4,6}^1 + R_{4,7}^1 + R_{4,8}^1,
\end{align*}
with
\begin{align*}
	R_{4,1}^1 & =\int_{\TT} \int_{\TT} \partial_\alpha^3 h(\alpha)   \frac{(\partial_\alpha^3 h)_-(\alpha)}{2} \frac{\sinh(h_-(\alpha))}{\sinh^2 \left( \frac{h_-(\alpha)}{2} \right)+ \sin^2\left(\frac{\beta}{2}\right)} h(\alpha-\beta) d \beta d \alpha
\end{align*}
\begin{align*}
	R_{4,2}^1& = \int_{\TT} \int_{\TT}  \partial_\alpha^3 h(\alpha)  3\frac{(\partial_\alpha^2 h)_-(\alpha) ( h_\alpha)_-(\alpha)}{2} \frac{\cosh(h_-(\alpha))}{\sinh^2 \left( \frac{h_-(\alpha)}{2} \right)+ \sin^2\left(\frac{\beta}{2}\right)} h(\alpha-\beta) d \beta d \alpha
\end{align*}
\begin{align*}
	R_{4,3}^1 & =- \int_{\TT} \int_{\TT}  \partial_\alpha^3 h(\alpha)  \frac{(\partial_\alpha^2 h)_-(\alpha) ( h_\alpha)_-(\alpha)}{4} \frac{\sinh^2(h_-(\alpha))}{ \left( \sinh^2 \left( \frac{h_-(\alpha)}{2} \right)+ \sin^2\left(\frac{\beta}{2}\right) \right)^2} h(\alpha-\beta) d \beta d \alpha
\end{align*}
\begin{align*}
	R_{4,4}^1 & =\int_{\TT} \int_{\TT}   \partial_\alpha^3 h(\alpha) \frac{((h_\alpha)_-(\alpha))^3}{2} \frac{\sinh(h_-(\alpha))}{\sinh^2 \left( \frac{h_-(\alpha)}{2} \right)+ \sin^2\left(\frac{\beta}{2}\right)}h(\alpha-\beta) d \beta d \alpha
\end{align*}
\begin{align*}
	R_{4,5}^1 & = -\int_{\TT} \int_{\TT}  \partial_\alpha^3 h(\alpha) \frac{((h_\alpha)_-(\alpha))^3}{4} \frac{\cosh(h_-(\alpha))\sinh(h_-(\alpha))}{\left( \sinh^2 \left( \frac{h_-(\alpha)}{2} \right)+ \sin^2\left(\frac{\beta}{2}\right) \right)^2}h(\alpha-\beta) d \beta d \alpha
\end{align*}
\begin{align*}
	R_{4,6}^1 & = - \int_{\TT} \int_{\TT}  \partial_\alpha^3 h(\alpha)  (h_\alpha)_-(\alpha) (\partial_\alpha^2 h)_-(\alpha)  \frac{\sinh^2(h_-(\alpha))}{\left( \sinh^2 \left( \frac{h_-(\alpha)}{2} \right)+ \sin^2\left(\frac{\beta}{2}\right) \right)^2}h(\alpha-\beta) d \beta d \alpha
\end{align*}
\begin{align*}
	R_{4,7}^1 & = -\int_{\TT} \int_{\TT}  \partial_\alpha^3 h(\alpha)   \frac{((h_\alpha)_-(\alpha))^3}{2}  \frac{\sinh(2h_-(\alpha))}{\left( \sinh^2 \left( \frac{h_-(\alpha)}{2} \right)+ \sin^2\left(\frac{\beta}{2}\right) \right)^2}h(\alpha-\beta) d \beta d \alpha
\end{align*}
\begin{align*}
	R_{4,8}^1& =\int_{\TT} \int_{\TT}   \partial_\alpha^3 h(\alpha)  \frac{((h_\alpha)_-(\alpha))^3}{2}  \frac{\sinh^3(h_-(\alpha))}{\left( \sinh^2 \left( \frac{h_-(\alpha)}{2} \right)+ \sin^2\left(\frac{\beta}{2}\right) \right)^3}h(\alpha-\beta) d \beta d \alpha
\end{align*}
We focus on the most singular term is $R_{4,1}^1$. We split $R_{4,1}^1$ into parts, in order to localize the singularity:
\begin{align*}
	R_{4,1}^1 = R_{4,1,1}^1 + R_{4,1,2}^1 + R_{4,1,3}^1 +R_{4,1,4}^1 +R_{4,1,5}^1 +R_{4,1,6}^1, 
\end{align*}
where 
\begin{align*}
	R_{4,1,1}^1 &= \int_{\TT} \int_{\TT} \partial_\alpha^3 h(\alpha)   \frac{(\partial_\alpha^3 h)_-(\alpha)}{2} \frac{\sinh(h_-(\alpha))-h_-(\alpha)}{\sinh^2 \left( \frac{h_-(\alpha)}{2} \right)+ \sin^2\left(\frac{\beta}{2}\right)} h(\alpha-\beta) d \beta  d \alpha \\
	& \lesssim \norm{h}_{{ H}^3}^2 \cosh(2 \norm{h}_{L^\infty}) \norm{h_\alpha}_{L^\infty}^2 \norm{h}_{L^\infty}^2.
\end{align*}
Similarly, using
$$
\left(\frac{\beta^2}{\sinh^2 \Big( \frac{h_-(\alpha)}{2} \Big)+ \sin^2\left(\frac{\beta}{2}\right)} - \frac{4}{1+ |h_\alpha(\alpha)|^2}\right)\lesssim \norm{h_\alpha}_{L^\infty}\norm{h_{\alpha\alpha}}_{L^\infty}\beta,
$$
we find that
\begin{align*}
	R_{4,1,2}^1 &= \int_{\TT} \int_{\TT} \partial_\alpha^3 h(\alpha)   \frac{(\partial_\alpha^3 h)_-(\alpha)}{2} \frac{h_-(\alpha)}{\beta^2} \left( \frac{\beta^2}{\sinh^2 \left( \frac{h_-(\alpha)}{2} \right)+ \sin^2\left(\frac{\beta}{2}\right)} - \frac{4}{1+ |h_\alpha(\alpha)|^2}\right) h(\alpha-\beta) d \beta d \alpha \\
	& \lesssim \norm{h_{\alpha\alpha}}_{L^\infty}\norm{h}_{{ H}^3}^2 \norm{h_\alpha}_{L^\infty}^2\norm{h}_{L^\infty}.
\end{align*}
We compute
\begin{align*}
	R_{4,1,3}^1 &= \int_{\TT} \int_{\TT}  \frac{4}{1+ |h_\alpha(\alpha)|^2} \partial_\alpha^3 h(\alpha)   \frac{(\partial_\alpha^3 h)_-(\alpha)}{2} \frac{h_-(\alpha)-\beta h_\alpha(\alpha)}{\beta^2}  h(\alpha-\beta) d \beta d \alpha \\
	& \lesssim  \norm{h}_{{ H}^3}^2 \norm{h}_{{\dot C}^2} \norm{h}_{L^\infty}.
\end{align*}
Furthermore, we can continue as follows
\begin{align*}
	R_{4,1,4}^1 &= \int_{\TT} \int_{\TT}  \frac{4}{1+ |h_\alpha(\alpha)|^2} \partial_\alpha^3 h(\alpha)   \frac{\partial_\alpha^3 h(\alpha)}{2}  h_\alpha(\alpha) h(\alpha-\beta) \left(\frac{1}{\beta} - \cot(\beta) \right) d \beta d \alpha \\
	& \quad + \int_{\TT} \int_{\TT}  \frac{4}{1+ |h_\alpha(\alpha)|^2} \partial_\alpha^3 h(\alpha)   \frac{\partial_\alpha^3 h(\alpha-\beta)}{2}  h_\alpha(\alpha) \frac{h(\alpha-\beta)-h(\alpha)}{\beta}d \beta d \alpha \\
	& \lesssim   \norm{h}_{{ H}^3}^2 \norm{h_\alpha}_{L^\infty} \norm{h}_{L^\infty} + \norm{h}_{{ H}^3}^2 \norm{h_\alpha}_{L^\infty}^2 
\end{align*}
and
\begin{align*}
	R_{4,1,5}^1 &= \int_{\TT} 2 \pi  \frac{4}{1+ |h_\alpha(\alpha)|^2} \partial_\alpha^3 h(\alpha)   \frac{\partial_\alpha^3 h(\alpha)}{2}  h_\alpha(\alpha) \mathcal{H}(h)(\alpha)  d \alpha \\
	& \quad + \int_{\TT} \int_{\TT}  \frac{4}{1+ |h_\alpha(\alpha)|^2} \partial_\alpha^3 h(\alpha)   \frac{\partial_\alpha^3 h(\alpha-\beta)}{2}  h_\alpha(\alpha) h(\alpha) \left( \frac{1}{\beta}- \cot(\beta) \right)d \beta d \alpha  \\
	& \lesssim \norm{h}_{{ H}^3}^2 \norm{h_\alpha}_{L^\infty} \norm{h}_{{{ H}^1}} + \norm{h}_{{ H}^3}^2 \norm{h_\alpha}_{L^\infty} \norm{h}_{L^\infty}. 
\end{align*}
Finally, we have that
\begin{align*}
	R_{4,1,6}^1 &= \int_{\TT} 2 \pi \frac{4}{1+ |h_\alpha(\alpha)|^2} \partial_\alpha^3 h(\alpha)   \frac{ \mathcal{H}(\partial_\alpha^3 h)(\alpha)}{2}  h_\alpha(\alpha) h(\alpha) d \alpha \\
	& \lesssim \norm{h}_{{ H}^3}^2 \norm{h_\alpha}_{L^\infty} \norm{h}_{L^\infty}.
\end{align*}
Consequently, 
\begin{equation*}
	R_4^1 \lesssim (1+t)^{-m} \mathcal{B}(\norm{h}_{L^\infty})  \tnorm{h} \mathcal{P}(\tnorm{h}),
\end{equation*}
and, as a consequence, 
\begin{equation*}
	\int_\TT \partial_\alpha^3 h(\alpha) \partial_\alpha^3 I_1(\alpha) d \alpha \lesssim  (1+t)^{-m} \mathcal{B}(\norm{h}_{L^\infty}) \tnorm{h} \mathcal{P} (\tnorm{h}).
\end{equation*}
The terms corresponding to the remaining integrals $I_2,I_3$ and $I_4$ can be estimated following the previous procedure and we conclude
\begin{equation}\label{H3estimate}
		\frac{d}{dt} \norm{h}_{ H^3}^2 \leq (1+t)^{-m \frac{2}{3}} C(\rho) \mathcal{B}_2(\tnorm{h})  \mathcal{P}_2 (\tnorm{h}),
\end{equation}
where  $\mathcal{P}_2$ is a polynomial with monomials with degree at least four and 
\begin{equation*}
	\mathcal{B}_2(x) = 1+\cosh^3(4x).
\end{equation*}
Equipped with this regularity and using standard energy arguments (see  \cite{gancedo_well-posedness_2021} section 6 for details), we conclude the continuity in time of $h$ in $H^3$.

\textbf{Global existence:} Collecting estimates \eqref{L2estimate} and the integrated (in time) version of \eqref{H3estimate}, we find 
\begin{equation}\label{tnormbound}
	\tnorm{h} \leq 2\norm{h_0}_{{ H}^3} + C(\rho) \mathcal{B}(\tnorm{h})  \mathcal{P} (\tnorm{h}),
\end{equation}
where $\mathcal{P}$ is a polynomial with monomials of degree at least 3, and 
\begin{equation*}
	\mathcal{B}(x) = 1+ \cosh^3(4x) .  
\end{equation*}
Let the initial data be such that
\begin{equation*}
	\norm{h_0}_{{ H}^3} < \delta.
\end{equation*}
Now we prove that, if $\delta$ is small enough, we find that
\begin{equation*}
	\tnorm{h} < 4 \delta.
\end{equation*}
We argue by contradiction. Assume that the solution reaches the value $\tnorm{h} = 4 \delta$ at certain time $t=T$. Then, by \eqref{tnormbound},
\begin{equation*}
	4 \delta  \leq 2\delta + C(\rho) \mathcal{B}(4 \delta )  \mathcal{P} (4\delta).
\end{equation*}
Using the smallness conditions of $\delta$ and the fact that $\mathcal{B}$ a monotonic increasing function, we find 
\begin{align*}
		4 \delta & < 2\delta + C\mathcal{B}(4\delta)  16\delta^2,\quad \mbox{and therefore} \quad 2\delta< C\mathcal{B}(1)16\delta^2,
\end{align*}
which is a contradiction if 
$$
\delta<\min\{(C\mathcal{B}(1)8)^{-1},4^{-1}\}.
$$
Consequently, 
\begin{equation*}
	\tnorm{h} < 4 \delta.
\end{equation*}
Finally, we can find a big enough constant such that 
$$ C \norm{h_0}_{H^3} \ge \delta,$$
thus 
$$ 4 \delta \leq 4C \norm{h_0}_{H^3}.$$
Therefore, abusing of notation for an arbitrary constant $C$, we have shown that as long as the initial data is small enough, we can find a classical solution to \eqref{eq:grafo2} such that 
\begin{equation}
	\tnorm{h} = \sup_{t \in [0,T]} \left( (1+t)^m \norm{h}_{L^2} + \norm{h}_{{ H}^3}  \right) \leq C \norm{h_0}_{H^3}
\end{equation}
for every $T>0$.
\end{proof}

\section{Global existence and exponential decay of analytic solutions  in the RT stable case}
This section is devoted to prove global well-posedness and exponential decay of analytic solutions to \eqref{eq:grafo2}. For that purpose, we will exploit the algebra structure of Wiener spaces $A^s_\nu$ (see \eqref{wiener} in Section \ref{section:linearoperator} for the definition of these spaces). We assume the RT stable scenario, where 
$$ \bar\rho=\frac{\rho^--\rho^+}{4} >0.$$
In this case, we will prove global existence and exponential decay of solutions for small enough initial data, given as a graph in a suitable Wiener space $$h_0 \in A^1_{\nu_0}(\TT), \text{ with } \nu_0 >0.$$ We recall that we can assume that the initial data has zero mean without loss of generality, and this property is propagated thanks to the existence of a stream function $\psi$. In particular, we prove the following result:
\begin{theo}[Global existence of solutions in Wiener algebras for small data] \label{globalwiener}
	Let $h_0\in {A}_{\nu_0}^1(\TT)$ be the initial data for \eqref{eq:grafo2} in the RT stable case,
	$$
\rho^- - \rho^+ >0.	
	$$ Assume that 
	\begin{equation*}
 \nu_0 >0.
	\end{equation*}
	There is a $0<\delta = \delta(\rho^--\rho^+,\nu_0)$ such that if
	$$
	\norm{h_0}_{  {A}_{\nu_0}^1}<\delta
	$$ 
	so that 
	$$ \nu_0 - \mathcal{M}(\norm{h_0}_{A_{\nu_0}^1}) >0 $$
	for a suitable non negative function $\mathcal{M}(x) \approx x + O(x^2) ,$
	there exists a unique global analytic solution $h(\alpha,t)$ 
	$$
	h\in  C([0,T];{A}_{\nu(t)}^1(\TT))
	$$
	of \eqref{eq:grafo2}  for arbitrary $T>0$ satisfying
	$$
	\norm{h}_{  {A}_{\nu(t)}^1        }\leq \norm{h_0}_{A^1_{\nu_0}}. 	
	$$	
\end{theo}
The details of the proof are shown in Subsection \ref{proof1}. Moreover, we can prove that, assuming the hypotheses of the previous theorem and certain extra conditions, we can get global existence and exponential decay of solutions in $A^0_{\nu^\ast}$, for a suitable $\nu^\ast < \nu_0$. This result is captured in the following statement:

\begin{theo}[Global existence and exponential decay of solutions for small data] \label{exponentialdecay}
	Let $h_0\in A^1_{\nu_0}(\TT)$ be the initial data for \eqref{eq:grafo2} in the RT stable case,
	$$
\rho^- - \rho^+ >0.	
	$$
	fulfilling the hypotheses of Theorem \ref{globalwiener}. In particular,
	$$ \norm{h_0}_{A^1_{\nu_0}} < \delta$$
	for a suitable $\delta =\delta(\rho^--\rho^+,\nu_0)$. Assume that 
	\begin{equation*}
		 0<\nu^\ast\leq \frac{\nu_0}{24}.
	\end{equation*}
	There is a $0< \varepsilon =  \varepsilon(\nu^\ast, \rho^--\rho^+,\norm{h_0}_{A^1_{\nu_0}}  )$ such that if
	$$
	\norm{h_0}_{ {A}^0_{\nu^\ast}  }<\varepsilon
	$$ 
	there exists a unique global analytic solution $h(\alpha,t)$ 
	$$
	h\in C([0,T]; {A}^0_{\nu^\ast} (\TT)  )
	$$
	of \eqref{eq:grafo2}  for arbitrary $T>0$ satisfying
	$$
	e^{\sqrt{((\rho^--\rho^+)/4) \nu^\ast t}}\norm{h}_{{A}^0}+ \norm{h}_{ {A}^0_{\nu^\ast}   }\leq C \norm{h_0}_{A^0_{\nu^\ast}}
	$$	
	for some constant $C>0$.
\end{theo}
The details of the proof are shown in Subsection \ref{proof2}.

\subsection{Fourier analysis}
In this subsection, we study the Fourier side of equation \eqref{eq:grafo2}, in order to exploit its structure in the Wiener Space context. We will use the equivalent expression given in \eqref{isolatelinear}.

First of all, we recall some properties of Wiener spaces. The fundamental property of Wiener spaces is their algebra structure:
\begin{equation*}
	\norm{f_1 \cdot f_2   \dots  f_n}_{A^s_{\nu}} \leq \prod_{i=1}^n \norm{f_i}_{A^s_{\nu}}.
\end{equation*}
We can also interpolate Wiener algebras using Hölder inequality in the following general way:
\begin{equation*}
	\norm{h(t)}_{{A}_{\nu(t)}^s} =  \sum_{k \in \ZZ} |k|^{\alpha s + (1-\alpha)s} |\hat{h}(k)| e^{(\beta \nu(t) + (1-\beta) \nu(t) ) |k| } \leq 	\norm{h(t)}_{{A}_{\beta \nu(t)p}^{\alpha s p}}^{1/p}  \norm{h(t)}_{{A}_{(1-\beta)\nu(t)q}^{(1-\alpha) s q}}^{1/q} ,
\end{equation*}
with $p,q \ge 1$ such that $\frac{1}{p} + \frac{1}{q} = 1$ and $0\leq \alpha,\beta \leq 1$.

Note that nonlinearities in the original equation will result in convolutions in the Fourier side. In the following, we will represent the convolution of $n$ copies of $f$ ($n-1$ convolutions), as 
\begin{equation*}
	\ast^n f(k) = (f \ast\underbrace{ \dots }_{n-1}\ast f)(k).
\end{equation*}
Also note that translations in the original variable $h$ will result in multipliers in the Fourier side. In this sense, 
\begin{equation*}
	\widehat{h(\cdot-\beta)} (k,\beta) = \hat{h}(k) e^{-ik\beta} \text{ and } 	\widehat {h_-} (k,\beta) = \hat{h}(k) m(k,\beta),
\end{equation*}
with 
\begin{equation*}
	m(k,\beta) = (1-e^{-ik\beta}) = -ik\beta \int_0^1 e^{-ik\beta (1-s)} ds.
\end{equation*}

Now, we take the $k-$th Fourier coefficient of equation \eqref{isolatelinear}: 
\begin{equation}\label{hfourier}
	\hat{h}_t(k,t) = - \bar \rho \frac{\hat h(k,t)}{|k|} + \frac{\bar \rho}{2 \pi} \left( \hat I_1(k,t)  + \hat I_2(k,t)  + \hat I_3(k,t)  + \hat I_4(k,t)   \right).
\end{equation}
In the following, we will drop the time dependence in the notation when it does not cause any ambiguity. Moreover, we consider that $h$ is small enough in order to  represent functions $\sinh(x)$ and $\log(1+x)$ by their Taylor series, and $ \frac{1}{1+x^2}$ as a geometric sum. Define
\begin{equation*}
	\mathcal{T}_1(\alpha,\beta) =\log  \left(1+\frac{\sinh^2 \left(\frac{h_-(\alpha)}{2}\right)}{\sin^2 \left(\frac{\beta}{2}\right)} \right)=  \sum_{n=1}^\infty \frac{(-1)^{n+1}  }{ \sin \left(\frac{\beta}{2}\right)^{2n} n}  \left( \sum_{j=0}^\infty \frac{(h_-(\alpha))^{2j+1}}{  2^{2j+1}(2j+1)!}  \right)^{2n}
\end{equation*}
and 
\begin{equation*}
	\mathcal{T}_2(\alpha,\beta ) =    \frac{h(\alpha-\beta) h_-(\alpha) }{2\sinh^2\left(\frac{h_-(\alpha)}{2}\right) + 2\sin^2\left(\frac{\beta}{2}\right)  } = \frac{1}{2 \sin^2 \left(\frac{\beta}{2}\right)} \sum_{n=0}^\infty (-1)^n \left( \frac{\sinh\left(\frac{h_-(\alpha)}{2}\right) }{\sin\left(\frac{\beta}{2}\right)}\right)^{2n}   h(\alpha-\beta) h_-(\alpha).
\end{equation*}
 Then, the nonlinear terms in \eqref{eq:grafo2} can be expressed in the following way:
\begin{equation*}
	I_1(\alpha) = \int_{\TT}  \mathcal{T}_1(\alpha,\beta) h(\alpha-\beta) d \beta,
\end{equation*}
\begin{equation*}
	I_2(\alpha) = \int_\TT \left( \mathcal{T}_1(\alpha,\beta) +\log \left(  \sin^2\left(\frac{\beta}{2}\right)     \right)\right) h(\alpha-\beta)  h_\alpha(\alpha) h_\alpha (\alpha-\beta) d\beta,
\end{equation*}
\begin{equation*}
	I_3(\alpha) =   \int_\TT \mathcal{T}_2(\alpha,\beta) \left[ (h_\alpha(\alpha)h_\alpha(\alpha-\beta)-1) \sinh(h_-(\alpha))\right] d\beta,
\end{equation*}
\begin{equation*}
	I_4(\alpha) =  \int_\TT \mathcal{T}_2(\alpha,\beta) \left[(h_\alpha (\alpha) + h_\alpha (\alpha-\beta) ) \sin (\beta)   \right] d\beta.
\end{equation*}
Therefore, their Fourier coefficients are
\begin{equation}\label{fourierI1}
	\hat I_1(k) =  \int_{\TT} \widehat{\mathcal{T}_1}(k,\beta) \ast e^{-ik\beta} \hat{h}(k) d \beta,
\end{equation}
\begin{align}
	\hat	I_2(k) & = \int_\TT \left(  \widehat {\mathcal{T}}_1(k,\beta) \right)  \ast \hat h(k) e^{-ik\beta} \ast (i k \hat h(k)) \ast (e^{-ik\beta} i k \hat h(k)) d\beta \nonumber \\
	& \quad  +  \int_\TT \log \left(  \sin^2\left(\frac{\beta}{2}\right)     \right)  \hat h(k) e^{-ik\beta} \ast (i k \hat h(k)) \ast (e^{-ik\beta} i k \hat h(k)) d\beta . \label{fourierI2}
\end{align}
\begin{align} 
	{\hat I}_3(k) &= \sum_{j=0}^\infty \frac{1}{(2j+1)!}  \int_\TT \widehat{\mathcal{T}}_2(k,\beta) \ast   (ik \hat h(k) \ast (ik e^{-ik\beta} \hat h(k)) \ast \ast^{2j+1} \left( m(k,\beta) \hat h(k) \right) d\beta \nonumber \\
	& \quad -  \sum_{j=0}^\infty \frac{1}{(2j+1)!}  \int_\TT \widehat{\mathcal{T}}_2(k,\beta) \ast  \ast^{2j+1} \left( m(k,\beta) \hat h(k) \right) d\beta. \label{fourierI3}
\end{align}
\begin{equation} \label{fourierI4}
	{\hat I}_4(k) =  \int_\TT \sin(\beta) \widehat{\mathcal{T}}_2(k,\beta) \ast \left[ (ik \hat{h}(k)) + (ik e^{-ik\beta} \hat h(k))  \right] d\beta.
\end{equation}

Taking the $k-$th Fourier coefficient, we find 
\begin{align}
	\widehat{	\mathcal{T}_1}(k,\beta) &= \sum_{n=1}^\infty \frac{(-1)^{n+1}  }{ n}  \ast^{2n} \left( \sum_{j=0}^\infty \frac{(h_-(\alpha))^{2j+1}}{ \sin \left(\frac{\beta}{2}\right) 2^{2j+1}(2j+1)!}  \right)^{\widehat{}}  \nonumber \\
	& = \sum_{n=1}^\infty \frac{(-1)^{n+1}  }{ n  \sin^{2n} \left(\frac{\beta}{2}\right)} \ast^{2n} \left( \sum_{j=0}^\infty  \frac{   \ast^{2j+1} (m(k,\beta) \hat{h}(k) ) }{ 2^{2j+1}(2j+1)!}       \right) \nonumber \\
	& = \sum_{n=1}^\infty \sum_{j=0}^\infty  \frac{(-1)^{n+1}  }{ n  \sin^{2n} \left(\frac{\beta}{2}\right)  2^{2j+1}(2j+1)!}  \ast^{2(j+n)+1} \left[m(k,\beta) \hat{h}(k)\right]  \nonumber \\
	& = \sum_{n=1}^\infty \sum_{j=0}^\infty  \frac{(-1)^{n+1}  }{ n  \sin^{2n} \left(\frac{\beta}{2}\right)  2^{2j+1}(2j+1)!} \nonumber \\ & \quad \times \sum_{k,k_1,k_2, \dots k_{2(j+n)} \neq 0}  \left(m(k_{2(j+n)}) m(k-k_1) \prod_{l=1}^{2(j+n)-1} m(k_l-k_{l+1}) \right) \nonumber \\
	& \quad \times \left(\hat h(k_{2(j+n)}) \hat{h}(k-k_1) \prod_{l=1}^{2(j+n)-1} \hat h(k_l-k_{l+1}) \right) . \label{T1}
\end{align}
Note that 
\begin{align*}
	& \left(m(k_{2(j+n)}) m(k-k_1) \prod_{l=1}^{2(j+n)-1} m(k_l-k_{l+1}) \right)   \\
	& =\int_0^1 ds \int_0^1 ds_1 \dots \int_0^1 ds_{{2(j+n)}}  (-i)^{2(j+n)+1}  \left( k_{2(j+n)} (k-k_1)\prod_{l=1}^{2(j+n)-1} (k_l-k_{l+1}) \right) \\
	& \quad \times \beta^{2n+2j+1}   \left( e^{-i k_{2(j+n)} \beta (1-s) } e^{-i (k-k_1)\beta (1-s)} \prod_{l=1}^{2(j+n)-1} e^{-i (k_l-k_{l+1})\beta (1-s_l)} \right).
\end{align*}

Similarly,
\begin{align}
	\widehat{\mathcal{T}}_2(k,\beta )  &= \left(\frac{1}{2 \sin^2 \left(\frac{\beta}{2}\right)} \sum_{n=0}^\infty \sum_{j=0}^{\infty} \frac{(-1)^n }{(2j+1)! 2^{2j+1}\sin^{2n} \left(\frac{\beta}{2}\right)} \ast^{2(n+j)+1} \left[ m(k) \hat h(k)  \right] \right) \nonumber \\
	& \quad  \ast (e^{-ik\beta} \hat h (k) ) \ast (m(k,\beta) \hat h(k))  \nonumber \\
	& = \left[  \frac{1}{2 \sin^2 \left(\frac{\beta}{2}\right)} \sum_{n=0}^\infty \sum_{j=0}^{\infty} \frac{(-1)^n }{(2j+1)! 2^{2j+1}\sin^{2n} \left(\frac{\beta}{2}\right)} \right. \nonumber \\
	& \quad \times  \sum_{k_1,k_2, \dots k_{2(j+n)} \in \ZZ}  \left(m(k_{2(j+n)}) m(k-k_1) \prod_{l=1}^{2(j+n)-1} m(k_l-k_{l+1}) \right) \nonumber \\
	& \left. \quad \times  \left(\hat h(k_{2(j+n)}) \hat{h}(k-k_1) \prod_{l=1}^{2(j+n)-1} \hat h(k_l-k_{l+1}) \right) \right] \ast (e^{-ik\beta} \hat h (k) ) \ast (m(k,\beta) \hat h(k)). \label{T2}
\end{align}

\subsection{Proof of Theorem \ref{globalwiener}} \label{proof1}

In this section, we prove Theorem \ref{globalwiener} exploiting the structure of nonlinear terms \eqref{fourierI1}-\eqref{fourierI4}. 

\begin{proof}
Consider the evolution of the norm 
\begin{equation*}
	\frac{d}{dt} \norm{h(t)}_{{A}_{\nu(t)}^0} = \nu'(t) \norm{h(t)}_{{A}_{\nu(t)}^1} + \sum_{k \neq 0} \frac{{\hat h_t(k,t)} { \overline{\hat h(k,t)}} +  {\hat h(k,t)} { \overline{\hat h_t(k,t)}}   }{2 {|\hat h (k,t)}|} e^{\nu(t) |k|}.
\end{equation*}
From \eqref{hfourier}, we find that
\begin{align}
	\frac{d}{dt} \norm{h(t)}_{{A}_{\nu(t)}^0} &\leq  \nu'(t) \norm{h(t)}_{{A}_{\nu(t)}^1} - \bar \rho \norm{h}_{{A}_{\nu(t)}^{-1}} +  \frac{\bar \rho}{2 \pi}\sum_{k \neq 0} \sum_{i=1}^4 \frac{{\hat I_i(k,t)} { \overline{\hat h(k,t)}} +  {\hat h(k,t)} { \overline{\hat I_i(k,t)}}   }{2 {|\hat h (k,t)}|} e^{\nu(t) |k|} \nonumber\\
	& \leq  \nu'(t) \norm{h(t)}_{{A}_{\nu(t)}^1} - \bar \rho \norm{h}_{{A}_{\nu(t)}^{-1}} +  \frac{\bar \rho}{2 \pi}\sum_{k \neq 0} \sum_{i=1}^4 |\hat{I}_i(k,t)| e^{\nu(t) |k|} \nonumber \\
	& \leq  \nu'(t) \norm{h(t)}_{{A}_{\nu(t)}^1} - \bar \rho \norm{h}_{{A}_{\nu(t)}^{-1}} +  \frac{\bar \rho}{2 \pi} \sum_{i=1}^4 \norm{ I_i}_{{{A}_{\nu(t)}^0}}. \label{A0nu}
\end{align}

Let us estimate the norms $\norm{I_i}_{A^0_{\nu(t)}}$. Note that 
\begin{equation*}
	|k| \leq |k-k_1| + |k_{2(n+j)}| \sum_{m=1}^{2(n+j)-1} |k_m-k_{m+1}|,
\end{equation*}
thus
\begin{equation*}
	e^{i \nu(t)|k| } \leq e^{i\nu(t) |k-k_1| } e^{i\nu(t) |k_{2n+2j}|}  \prod_{m=1}^{2(n+j)-1} e^{  i\nu(t) |k_m-k_{m+1}  |   }.
\end{equation*} 

Therefore, using \eqref{T1} and 
\begin{equation*}
\sum_{i=1}^\infty \frac{x^{2n}}{n} = -\log (1-x^2) = \log \left( \frac{1}{1-x^2}\right)
\end{equation*}
for $x$ small enough, it holds 
\begin{align}
	\norm{ \mathcal{T}_1}_{{{A}_{\nu(t)}^0}} & \leq  \sum_{k \neq 0} \sum_{n=1}^\infty \sum_{j=0}^\infty  \frac{ \beta^{2j+2n+1} }{ n  \sin^{2n} \left(\frac{\beta}{2}\right)  2^{2j+1}(2j+1)!} \seminorm{\ast^{2(j+n)+1} \left[ ik \hat h(k)\right] }e^{ \nu(t) |k|} \nonumber \\
	& \leq   \sum_{k \neq 0} \sum_{n=1}^\infty \sum_{j=0}^\infty  \frac{ \pi^{2n} \beta^{2j+1} }{ n  2^{2j+1}(2j+1)!} \seminorm{\ast^{2(j+n)+1} \left[ ik \hat h(k)\right] }e^{ \nu(t) |k|} \nonumber \\
	& \leq  \sum_{n=1}^\infty \sum_{j=0}^\infty  \frac{ \pi^{2n+2j+1}}{ n  2^{2j+1}(2j+1)!} \norm{h}_{{{A}_{\nu(t)}^1}}^{2(j+n)+1} \nonumber   \\
	& \leq \sinh  \left( \frac{\pi}{2}   \norm{h}_{{{A}_{\nu(t)}^1}}\right) \log \left( \frac{1}{1-\pi^2 \norm{h}_{{{A}_{\nu(t)}^1}}^2 }\right) \nonumber \\
	& \leq 	\mathcal{G}_1 \left(     \norm{h}_{{{A}_{\nu(t)}^1}} \right), \label{T1A0}
\end{align}
where we  defined 
\begin{equation*}
	\mathcal{G}_1 \left(     \norm{h}_{{{A}_{\nu(t)}^1}} \right) = \sinh  \left( \frac{\pi}{2}   \norm{h}_{{{A}_{\nu(t)}^1}}\right) \log \left( \frac{1}{1-\pi^2 \norm{h}_{{{A}_{\nu(t)}^1}}^2 }\right)  .
\end{equation*}

From \eqref{fourierI1} and \eqref{T1A0}, we get the estimates
\begin{align}
	\norm{ I_1}_{{{A}_{\nu(t)}^0}}  & \leq \int_{\TT} 	 \norm{ \mathcal{T}_1}_{{{A}_{\nu(t)}^0}} 	 \norm{ h}_{{{A}_{\nu(t)}^0}} d \beta \nonumber\\
	& \leq 2 \pi 	\mathcal{G}_1 \left(     \norm{h}_{{{A}_{\nu(t)}^1}} \right)  \norm{ h}_{{{A}_{\nu(t)}^0}}  \label{I1A0}
\end{align}
and
\begin{align}
	\norm{ I_2}_{{{A}_{\nu(t)}^0}}  & \leq \int_{\TT} \left(	 \norm{ \mathcal{T}_1}_{{{A}_{\nu(t)}^0}} +\log \left(  \sin^2\left(\frac{\beta}{2}\right)     \right)\right)	 \norm{ h}_{{{A}_{\nu(t)}^1}}^2 \norm{ h}_{{{A}_{\nu(t)}^0}} d \beta \nonumber \\
	& \leq 2\pi 	\mathcal{G}_1 \left(     \norm{h}_{{{A}_{\nu(t)}^1}} \right)  \norm{ h}_{{{A}_{\nu(t)}^0}}  \norm{ h}_{{{A}_{\nu(t)}^1}}^2 + 4 \pi \log(2) \norm{ h}_{{{A}_{\nu(t)}^0}}  \norm{ h}_{{{A}_{\nu(t)}^1}}^2. \label{I2A0}
\end{align}
From \eqref{fourierI2},  \eqref{T2} and 
\begin{equation*}
	 \sum_{i=0}^\infty x^{2n} = \frac{1}{1-x^2}
\end{equation*}
for $x$ small enough, we get 
\begin{align}
	\norm{	{\mathcal{T}}_2 }_{{{A}_{\nu(t)}^0}}  &\leq \sum_{k \neq 0}  \frac{1}{2 \sin^2 \left(\frac{\beta}{2}\right)} \sum_{n=0}^\infty \sum_{j=0}^{\infty} \frac{\beta^{2n+2j+2} }{(2j+1)! 2^{2j+1}\sin^{2n} \left(\frac{\beta}{2}\right)}\seminorm{ \ast^{2(n+j)+2} \left[ ik \hat h(k)  \right] \ast \hat h(k)} e^{i\nu(t)|k|} \nonumber \\
	& \leq  \frac{1}{2} \sum_{n=0}^\infty \sum_{j=0}^{\infty} \frac{\pi^{2(n+j)+2} }{(2j+1)! 2^{2j+1} } \norm{h}^{2(n+j)+2}_{{{A}_{\nu(t)}^1}} \norm{h}_{{{A}_{\nu(t)}^0}} \nonumber \\
	& \leq  \frac{\pi}{2} \sinh \left( \frac{\pi}{2} \norm{	h}_{{{A}_{\nu(t)}^1}}    \right) \frac{1}{1- \pi^2 \norm{	h}^2_{{{A}_{\nu(t)}^1}}  } \norm{h}_{{{A}_{\nu(t)}^1}} \norm{h}_{{{A}_{\nu(t)}^0}}   \nonumber  \\
	& \leq \mathcal{G}_2(  \norm{h}_{{{A}_{\nu(t)}^1}}   )\norm{h}_{{{A}_{\nu(t)}^0}} , \label{T2A0}
\end{align}
where we defined 
\begin{equation*}
	\mathcal{G}_2(  \norm{h}_{{{A}_{\nu(t)}^1}}   ) =  \frac{\pi}{2} \sinh \left( \frac{\pi}{2} \norm{	h}_{{{A}_{\nu(t)}^1}}    \right) \frac{1}{1- \pi^2 \norm{	h}^2_{{{A}_{\nu(t)}^1}}  }  \norm{	h}_{{{A}_{\nu(t)}^1}}.
\end{equation*}

From  \eqref{fourierI3},  \eqref{fourierI4}   and \eqref{T2A0}, we get the estimates
\begin{align}
	\norm{ I_3}_{{{A}_{\nu(t)}^0}}  & \leq \int_{\TT} 	 \norm{ \mathcal{T}_2}_{{{A}_{\nu(t)}^0}} \left( 	\norm{ h}_{{{A}_{\nu(t)}^1}}^2+1   \right)  \sinh\left( \pi 	\norm{ h}_{{{A}_{\nu(t)}^1}}   \right)	d \beta \nonumber \\
	& \leq  2\pi  	 \mathcal{G}_2(  \norm{h}_{{{A}_{\nu(t)}^1}}   ) \norm{h}_{{{A}_{\nu(t)}^0}}   \left( 	\norm{ h}_{{{A}_{\nu(t)}^1}}^2+1   \right) \sinh \left(  \pi   	\norm{ h}_{{{A}_{\nu(t)}^1}}  \right) \label{I3A0}
\end{align}
and
\begin{align}
	\norm{ I_4}_{{{A}_{\nu(t)}^0}}  & \leq \int_{\TT} 2 |\sin(\beta)| 	 \norm{ \mathcal{T}_2}_{{{A}_{\nu(t)}^0}} 	\norm{ h}_{{{A}_{\nu(t)}^1}} d \beta \nonumber \\
	& \leq 8 	 \mathcal{G}_2(  \norm{h}_{{{A}_{\nu(t)}^1}}   ) \norm{h}_{{{A}_{\nu(t)}^0}} \norm{h}_{{{A}_{\nu(t)}^1}} . \label{I4A0}
\end{align}

Joining the estimates of terms $\norm{I_i}_{A^0_{\nu(t)}}$ in \eqref{A0nu},
\begin{equation}\label{controlA0}
	\frac{d}{dt} \norm{h(t)}_{{A}_{\nu(t)}^0}  \leq	 \nu'(t) \norm{h(t)}_{{A}_{\nu(t)}^1} - \bar \rho \norm{h}_{{A}_{\nu(t)}^{-1}} +  C_1(\bar \rho ) \norm{h}_{{A}_{\nu(t)}^{1}} \norm{h}_{{A}_{\nu(t)}^{0}} \mathcal{J}_1( \norm{h}_{{A}_{\nu(t)}^{1}}   ),
\end{equation}
where  	$\mathcal{J}_1$ is an increasing function and $C(\bar \rho)$ depends linearly on the density jump. In fact, 
\begin{equation}\label{asymptoticJ1}
	\mathcal{J}_1(x) \approx x + O(x^2)
\end{equation}

The evolution of the ${A}^1_{\nu(t)}$ norm can be expressed, similarly, as
\begin{align*}
	\frac{d}{dt} \norm{h(t)}_{{A}_{\nu(t)}^1}	& \leq  \nu'(t) \norm{h(t)}_{{A}_{\nu(t)}^2} - \bar \rho \norm{h}_{{A}_{\nu(t)}^{0}} +  \frac{\bar \rho}{2 \pi} \sum_{i=1}^4 \norm{ I_i}_{{{A}_{\nu(t)}^1}}.
\end{align*}

We have 
\begin{align}
	\norm{ \mathcal{T}_1}_{{{A}_{\nu(t)}^1}} 
	& \leq   \sum_{k \neq 0} \sum_{n=1}^\infty \sum_{j=0}^\infty  \frac{ \pi^{2n} \beta^{2j+1} |k|}{ n  2^{2j+1}(2j+1)!} \seminorm{\ast^{2(j+n)+1} \left[ ik \hat h(k)\right] }e^{ \nu(t) |k|}  \nonumber  \\
	& \leq   \sum_{k \neq 0}  \sum_{n=1}^\infty \sum_{j=0}^\infty (2(j+n)+1) \frac{ \pi^{2n+2j+1}}{ n  2^{2j+1}(2j+1)!}\seminorm{\ast^{2(j+n)} \left[ ik \hat h(k)\right] \ast \left[|k|^2 i \hat{h}(k)\right] }e^{ \nu(t) |k|}  \nonumber\\
	& \leq   \sum_{n=1}^\infty \sum_{j=0}^\infty (2(j+n)+1) \frac{ \pi^{2n+2j+1}}{ n  2^{2j+1}(2j+1)!} \norm{h}_{{{A}_{\nu(t)}^1}}^{2(j+n) } \norm{h}_{{{A}_{\nu(t)}^2}}  \nonumber  \\
	& \leq   \norm{h}_{{{A}_{\nu(t)}^2}} \frac{\pi}{2} \cosh \left( \frac{\pi}{2}  \norm{h}_{{{A}_{\nu(t)}^1}} \right) \log \left(  \frac{1}{1-\pi^2 \norm{h}_{{{A}_{\nu(t)}^1}}^2 }\right) \nonumber \\
	& \quad +  \norm{h}_{{{A}_{\nu(t)}^2}}\sinh \left( \frac{\pi}{2}  \norm{h}_{{{A}_{\nu(t)}^1}} \right)  \frac{2 \pi^2 \norm{h}_{{{A}^1_{\nu(t)}}} }{1-\pi^2 \norm{h}_{{{A}_{\nu(t)}^1}}^2 } \nonumber \\
	& \leq \mathcal{G}_1'(  \norm{h}_{{{A}_{\nu(t)}^1}}    ) \norm{h}_{{{A}_{\nu(t)}^2}} \label{T1A1}
\end{align}
From \eqref{T1A1}, we get 
\begin{align*}
	\norm{ I_1}_{{{A}_{\nu(t)}^1}}  & \leq \int_{\TT} 	\norm{ \mathcal{T}_1(k,\beta) \ast e^{-ik\beta} \hat h (k)}_{{{A}_{\nu(t)}^1}} d\beta\\
	& \leq \int_{\TT}  	\left(\norm{ \mathcal{T}_1}_{{{A}_{\nu(t)}^1}}  	\norm{ h}_{{{A}_{\nu(t)}^0}} + \norm{ \mathcal{T}_1}_{{{A}_{\nu(t)}^0}}  	\norm{ h}_{{{A}_{\nu(t)}^1}}   \right)   d\beta\\
	& \leq 2 \pi \mathcal{G}_1'(  \norm{h}_{{{A}_{\nu(t)}^1}}    ) \norm{h}_{{{A}_{\nu(t)}^2}} \norm{h}_{{{A}_{\nu(t)}^0}} + 2 \pi \mathcal{G}_1(  \norm{h}_{{{A}_{\nu(t)}^1}}    ) \norm{h}_{{{A}_{\nu(t)}^1}}    \\
	& \leq C \norm{ h}_{{{A}_{\nu(t)}^0}} \norm{ h}_{{{A}_{\nu(t)}^2}} F_1 \left(  \norm{ h}_{{{A}_{\nu(t)}^1}}   \right)
\end{align*}
and
\begin{align*}
	\norm{ I_2}_{{{A}_{\nu(t)}^1}} &\leq  \int_\TT  \norm{ \widehat {\mathcal{T}}_1(k,\beta)  \ast \hat h(k) e^{-ik\beta} \ast (i k \hat h(k)) \ast (e^{-ik\beta} i k \hat h(k))}_{{{A}_{\nu(t)}^1}}  d\beta \\
	& \quad + \int_\TT \seminorm{ \log \left(  \sin^2\left(\frac{\beta}{2}\right)     \right) } \norm{  \hat h(k) e^{-ik\beta} \ast (i k \hat h(k)) \ast (e^{-ik\beta} i k \hat h(k))}_{{{A}_{\nu(t)}^1}}  d\beta \\
	& \leq 	\norm{ \mathcal{T}_1}_{{{A}_{\nu(t)}^1}} 	\norm{ h}_{{{A}_{\nu(t)}^0}} 	\norm{ h}_{{{A}_{\nu(t)}^1}}^2 +  \norm{ \mathcal{T}_1}_{{{A}_{\nu(t)}^0}} 	\norm{ h}_{{{A}_{\nu(t)}^1}}^3 +  2 \norm{ \mathcal{T}_1}_{{{A}_{\nu(t)}^0}} 	\norm{ h}_{{{A}_{\nu(t)}^0}} \norm{ h}_{{{A}_{\nu(t)}^1}} \norm{ h}_{{{A}_{\nu(t)}^2}} \\
	& \leq C \norm{ h}_{{{A}_{\nu(t)}^0}} \norm{ h}_{{{A}_{\nu(t)}^2}} F_2 \left(  \norm{ h}_{{{A}_{\nu(t)}^1}}   \right).
\end{align*}
Furthermore,
\begin{equation} \label{T2A1}
	\norm{	{\mathcal{T}}_2 }_{{{A}_{\nu(t)}^1}} 
	\leq   \mathcal{G}_2'(  \norm{h}_{{{A}_{\nu(t)}^1}}   )\norm{h}_{{{A}_{\nu(t)}^0}} \norm{h}_{{{A}_{\nu(t)}^2}} +  \mathcal{G}_2(  \norm{h}_{{{A}_{\nu(t)}^1}}   )\norm{h}_{{{A}_{\nu(t)}^1}}.
\end{equation}
From \eqref{T2A1}, we get 
\begin{align*}
	\norm{ I_3}_{{{A}_{\nu(t)}^1}}  & \leq C \norm{ h}_{{{A}_{\nu(t)}^0}} \norm{ h}_{{{A}_{\nu(t)}^2}} F_3 \left(  \norm{ h}_{{{A}_{\nu(t)}^1}}   \right).
\end{align*}
and 
\begin{align*}
	\norm{ I_4}_{{{A}_{\nu(t)}^1}}  & \leq  C \norm{ h}_{{{A}_{\nu(t)}^0}} \norm{ h}_{{{A}_{\nu(t)}^2}} F_4 \left(  \norm{ h}_{{{A}_{\nu(t)}^1}}   \right),
\end{align*}
where the functions $F_1, \dots F_4$ are increasing functions.

Joining all the estimates of terms $\norm{I_i}_{A^1_{\nu(t)}}$ together, we get 
\begin{equation}\label{controlA1}
	\frac{d}{dt} \norm{h(t)}_{{A}_{\nu(t)}^1} \leq	 \nu'(t) \norm{h(t)}_{{A}_{\nu(t)}^2} - \bar \rho \norm{h}_{{A}_{\nu(t)}^{0}} +  C_2(\bar \rho ) \norm{ h}_{{{A}_{\nu(t)}^0}} \norm{ h}_{{{A}_{\nu(t)}^2}} \mathcal{J}_2 \left(  \norm{ h}_{{{A}_{\nu(t)}^1}}   \right),
\end{equation}
where $\mathcal{J}_2$ is, again, an increasing function, and $C_2(\bar \rho)$ depends linearly on the density jump.

\textbf{Control of norms for all time.} Define 
\begin{equation}\label{dernu}
	\nu'(t) = - \left( \max_{i=1,2} \{ C_i(\bar\rho) \mathcal{J}_i  \left(  \norm{ h}_{{{A}_{\nu(t)}^1}}   \right)  \} + \varepsilon \right)  \norm{ h}_{{{A}_{\nu(t)}^0}}. 
\end{equation}
Then, by \eqref{controlA0}, \eqref{controlA1} and \eqref{dernu}
\begin{equation*}
	\frac{d}{dt} \norm{h(t)}_{{A}_{\nu_0}^0} < 0 ,
\end{equation*}
\begin{equation*}
	\frac{d}{dt} \norm{h(t)}_{{A}_{\nu_0}^1} < 0 ,
\end{equation*}
for every $ t \ge 0$. Consequently
\begin{equation}\label{decayA0A1}
	\norm{h(t)}_{{A}_{\nu(t)}^0} \leq  \norm{h_0}_{{A}_{\nu_0}^0} \text{ and }  \norm{h(t)}_{{A}_{\nu(t)}^1} \leq  \norm{h_0}_{{A}_{\nu_0}^1}
\end{equation}
for every $t \ge 0$.
In addition, from \eqref{controlA1}, taking into account the damping term, we have
\begin{equation}\label{controlA12}
	\norm{h(t)}_{{A}_{\nu(t)}^1} + \bar \rho \int_0^t  \norm{ h(s)}_{{{A}_{\nu(s)}^0}} ds \leq   \norm{h_0}_{{A}_{\nu_0}^1} .
\end{equation}
Consequently, the solution is controlled by the initial data for all time $t \ge 0$.


\textbf{Control of the analiticity band.} 
Using \eqref{controlA12} and the monotonic increasing character of functions $\mathcal{J}_i$,
\begin{align*}
	\nu(t) &= \nu_0 + \int_0^t \nu'(t) dt  \\
	&  \ge \nu_0  - \left( \max_{i=1,2} \{ C_i(\bar\rho) \mathcal{J}_i  \left(  \norm{ h_0}_{{{A}_{\nu_0}^1}}   \right)  \} + \varepsilon \right) \int_0^t  \norm{ h(s)}_{{{A}_{\nu(s)}^0}} ds \\
	& \ge \nu_0 - \frac{1}{\bar \rho} \left( \max_{i=1,2} \{ C_i(\bar\rho) \mathcal{J}_i  \left(  \norm{ h_0}_{{{A}_{\nu_0}^1}}   \right)  \} + \varepsilon \right)  \norm{h_0}_{{A}_{\nu_0}^1} \\
	& > 0 
\end{align*}
for small enough data $\norm{h_0}_{{A}_{\nu_0}^1}$. This ensures that the analiticity bound does not collapse at any time $t \ge 0$.

In particular, we can choose small enough initial data such that
\begin{equation}\label{analband}
0 <	\frac{\nu_0}{2} \leq \nu(t) \leq \nu_0.
\end{equation}

The estimates \eqref{decayA0A1} and \eqref{analband} give us the global existence of solutions in $A^1_{\nu(t)}$.
\end{proof}

\subsection{Proof of Theorem \ref{exponentialdecay}}  \label{proof2}
In this section, we prove Theorem \ref{exponentialdecay}. The proof relies on the proof of Theorem \ref{globalwiener} and a similar strategy to the one used in the proof of Theorem \ref{stabilitythm}: we exploit the decay properties of the semi-group in Wiener algebras and combine them with a priori estimates.
\begin{proof}
	Before starting the argument of the proof, let us establish the following decay lemma: 
	\begin{lem}\label{lemexp}
		Let $a,b >0$. Then, 
		\begin{equation*}
			\int_0^t e^{-\sqrt{  a (t-s)} } e^{-\sqrt{b s}} ds \leq e^{-\sqrt {  \frac{\min \{a,b\}}{2} t  } } .
		\end{equation*}
	\end{lem}
	\begin{proof}
		Define $m = \min \{ a,b\}$. Then,
		\begin{align*}
			\int_0^t e^{-\sqrt{  a (t-s)} } e^{-\sqrt{ b s}} ds & = 	\int_0^{\frac{t}{2}} e^{-\sqrt{ a (t-s)} } e^{-\sqrt{ b s}} ds + 	\int_{\frac{t}{2}}^t e^{-\sqrt{  a (t-s)} } e^{-\sqrt{ b s}} ds \\
			& \leq \int_0^{\frac{t}{2}} e^{-\sqrt{ m (t-s)} } e^{-\sqrt{ m s}} ds \\ & \quad + 	\int_{\frac{t}{2}}^t e^{-\sqrt{  m (t-s)} } e^{-\sqrt{m s}} ds \\
			& \leq e^{-\sqrt{ m \frac{t}{2}}}  \int_0^{\frac{t}{2}} e^{-\sqrt{ m s}} ds  \\
			& \quad +  e^{-\sqrt{ m \frac{t}{2}}}	\int_{\frac{t}{2}}^t e^{-\sqrt{  m (t-s)} }ds \\
			& \leq   e^{-\sqrt{ m \frac{t}{2}    }}  \frac{2}{   m }  \left(  1- e^{-\sqrt{ m t   }} \left( 1 + \sqrt{ m t }  \right)     \right) \\
			& \leq c( m) e^{-\sqrt{ m \frac{t}{2}    }} .
		\end{align*}
	\end{proof}

	Let $h \in C([0,T]; A^1_{\nu(t)})$ be a solution of \eqref{eq:grafo2} in the RT stable case constructed under the hypotheses of Theorem \ref{globalwiener}, with small enough initial data
	$$ \norm{h_0}_{A^1_{\nu_0}} < \delta.$$
	Define $\nu^\ast$ in such a way that 
	\begin{equation*}
		0 < 12	\nu^\ast \leq \frac{\nu_0}{2} \leq \nu(t) \leq \nu_0,
	\end{equation*}
and define the norm 
\begin{equation}\label{tnorm2}
	\tnorm{h}_{\nu^\ast} = \sup_{t \in [0,T]}  \left(e^{\sqrt{\bar\rho \nu^\ast t}} \norm{h}_{{A}^0} +  \norm{h}_{{A}^0_{\nu^\ast}} \right).
\end{equation}
The aim is to produce uniform bounds of \eqref{tnorm2}, only depending on the size of the initial data $h_0$.

Using the estimate \eqref{lineardecay2A} in Proposition \ref{PropDH} for $\nu = 2\nu^\ast$ and Duhamel formula, we get 
\begin{equation*}
	\norm{h}_{{A}^0}  \lesssim e^{-\sqrt{\bar \rho 2 \nu^\ast t}} 	\norm{h_0}_{{A}^0_{2\nu^\ast}} + \frac{\bar \rho}{2\pi} \int_0^t e^{-\sqrt{\bar \rho 2 \nu^\ast (t-s)}} \norm{I_1 + I_2 + I_3 + I_4}_{{A}_{ 2\nu^\ast}^0} ds.
\end{equation*}

Using $x \leq c(\nu^\ast) e^{2 \nu^\ast x}$ and Hölder inequality, we can easily derive the estimates
\begin{align*}
	\norm{h}_{{A}^1_{2\nu^\ast}} &= \sum_{k \neq 0} |k| e^{2 \nu^\ast |k|} |\hat{h}(k)| \\
	&\leq c(\nu^\ast) \norm{h}_{{A}^0_{3\nu^\ast}} \\
	& \leq c(\nu^\ast) \norm{h}_{{A}^0}^{3/4} \norm{h}_{ {{A}}^0_{12\nu^\ast}}^{1/4}
\end{align*}
and
\begin{align*}
	\norm{h}_{{A}^0_{2\nu^\ast}} & \leq c(\nu^\ast)  \norm{h}_{{A}^0}^{1/2} \norm{h}_{{A}^0_{4\nu^\ast}}^{1/2}.
\end{align*}
Since $12 \nu^\ast \leq \frac{\nu_0}{2} \leq \nu(t)$, we have the uniform estimates 
$$\norm{h}_{A^0_{4\nu^\ast}} \leq \norm{h}_{A^0_{12\nu^\ast}} \leq \norm{h}_{A^0_{\nu(t)}} \leq \norm{h_0}_{A^0_{\nu_0}}  $$
and 
$$\norm{h}_{A^1_{2\nu^\ast}} \leq \norm{h}_{A^1_{\nu(t)}} \leq \norm{h_0}_{A^1_{\nu_0}},  $$
 so that, by \eqref{I1A0},\eqref{I2A0}, \eqref{I3A0} and \eqref{I4A0} in a band of width $2 \nu^\ast$ and the estimates above,
\begin{align*}
 \norm{I_1 + I_2 + I_3 + I_4}_{{A}_{2\nu^\ast}^0} & \leq c(\bar \rho ) \norm{h}_{{A}_{2\nu^\ast}^{1}} \norm{h}_{{A}_{2\nu^\ast}^{0}} \mathcal{J}_1( \norm{h}_{{A}_{2\nu^\ast}^{1}}   ) \\
 & \leq c(\bar \rho ) \norm{h}_{{A}_{2\nu^\ast}^{1}}^2 \norm{h}_{{A}_{2\nu^\ast}^{0}} F( \norm{h}_{{A}_{2\nu^\ast}^{1}}   ) \\
 & \leq c(\bar \rho) \norm{h}_{{A}^0}^2 \norm{h}_{{A}^0_{\nu(t)}} F( \norm{h}_{{A}_{\nu(t)}^{1}}   ) \\
 & \leq c(\bar \rho,\norm{h_0}_{A^1_{\nu_0}}) e^{-2\sqrt{\bar \rho  \nu^\ast s }} \tnorm{h}_{\nu^\ast}^2,
\end{align*}
where the first inequality comes from \eqref{asymptoticJ1}.

Consequently, applying Lemma \ref{lemexp},
\begin{align*}
	\norm{h}_{{A}^0} &  \leq e^{-\sqrt{\bar \rho 2 \nu^\ast t}} 	\norm{h_0}_{{A}^0_{2\nu^\ast}} + c(\nu^\ast,\bar \rho, \norm{h_0}_{A^1_{\nu_0}}) \tnorm{h}^2_{\nu^\ast} \int_0^t e^{-\sqrt{\bar \rho 2 \nu^\ast (t-s)}} e^{-\sqrt{\bar\rho 4 \nu^\ast s}} ds \\
	& \leq  e^{-\sqrt{\bar \rho  \nu^\ast t}} 	\norm{h_0}_{{A}^0_{2\nu^\ast}} + c(\nu^\ast,\bar \rho, \norm{h_0}_{A^1_{\nu_0}}) \tnorm{h}^2_{\nu^\ast}   e^{-\sqrt{\bar\rho  \nu^\ast t}}.
\end{align*}
Thus, 
\begin{equation}\label{decayestimate}
e^{\sqrt{\bar \rho \nu^\ast t}} 	\norm{h}_{{A}^0}   \leq \norm{h_0}_{{A}^0_{2\nu^\ast}} +c(\nu^\ast,\bar \rho, \norm{h_0}_{A^1_{\nu_0}}) \tnorm{h}^2_{\nu^\ast}  .
\end{equation}

Furthermore, by Hölder inequality, we can similarly derive the estimates 
\begin{align*}
	\norm{h}_{{A}^1_{\nu^\ast}} &= \sum_{k \neq 0} |k| e^{ \nu^\ast |k|} |\hat{h}(k)| \\
	&\leq c(\nu^\ast) \norm{h}_{{A}^0_{2\nu^\ast}} \\
	& \leq c(\nu^\ast) \norm{h}_{{A}^0}^{3/4} \norm{h}_{ {{A}}^0_{8\nu^\ast}}^{1/4}
\end{align*}
and
\begin{align*}
	\norm{h}_{{A}^0_{\nu^\ast}} & \leq c(\nu^\ast)  \norm{h}_{{A}^0}^{1/2} \norm{h}_{{A}^0_{2\nu^\ast}}^{1/2}.
\end{align*}

By \eqref{controlA0} in a band of width $ \nu^\ast$ and \eqref{asymptoticJ1}, it holds that
\begin{align*}
	\frac{d}{dt}\norm{h}_{{A}^0_{\nu^\ast}} &\leq  C_1(\bar \rho ) \norm{h}_{{A}_{\nu^\ast}^{1}}^2 \norm{h}_{{A}_{\nu^\ast}^{0}} F( \norm{h}_{{A}_{\nu^\ast}^{1}}   ) \\
	& \leq c(\nu^\ast,\bar \rho, \norm{h_0}_{A^1_{\nu_0}}) e^{-2 \sqrt{\bar \rho \nu^\ast t}} \tnorm{h}^2_{\nu^\ast} .
\end{align*}
Integrating from $0$ to $t$, we get 
\begin{align}
	\norm{h}_{{A}^0_{\nu^\ast}} &\leq 	\norm{h_0}_{{A}^0_{\nu^\ast}} + c(\nu^\ast,\bar \rho, \norm{h_0}_{A^1_{\nu_0}}) \tnorm{h}^2_{\nu^\ast}  \int_0^t e^{- \sqrt{\bar \rho 4\nu^\ast s}} ds \nonumber \\
	& \leq \norm{h_0}_{{A}^0_{\nu^\ast}} + c(\nu^\ast,\bar \rho, \norm{h_0}_{A^1_{\nu_0}}) \tnorm{h}^2_{\nu^\ast} . \label{analyticityestimate}
\end{align}

Joining the estimates \eqref{decayestimate} and \eqref{analyticityestimate}, we find 
\begin{equation*} 
	\tnorm{h}_{\nu^\ast}  \leq \norm{h_0}_{{A}^0_{\nu^\ast}} + c(\nu^\ast,\bar \rho, \norm{h_0}_{A^1_{\nu_0}}) \tnorm{h}^2_{\nu^\ast} .
\end{equation*}

Assuming the condition
\begin{equation*}
	\norm{h_0}_{A^0_{\nu^\ast}} < \varepsilon,
\end{equation*}
for $\varepsilon >0$ small enough, we can reproduce the argument in the proof of Theorem \ref{stabilitythm} and get 

\begin{equation}
	\tnorm{h}_{\nu^\ast}  = \sup_{t \in [0,T]} \left( e^{\sqrt{\bar \rho \nu^\ast t }} \norm{h}_{{A}^0} + \norm{h}_{ {A}^0_{\nu^\ast}  }  \right) \leq C \norm{h_0}_{A^0_{\nu^\ast}}
\end{equation}
for every $T>0$. This gives us the global existence of analytic interfaces in $A^0_{\nu^\ast}$ and the exponential decay of the $A^0$ norm in the stable regime.

\end{proof}

\section{Exponential growth of solutions in the RT unstable case}

In this section, we prove that in the unstable regime of the densities, where the denser fluid lies above the less dense, smooth solutions grow exponentially in certain Wiener norms. We collect this result in the following theorem:

\begin{theo}[Exponential growth of solutions in the RT unstable case for small data] \label{exponentialgrowth}
	Let $T>0$ be an arbitrary fixed parameter. Then, it exists a family of smooth initial data 
	$$ g_0 \in A^0_{\nu\ast} (\TT)  $$
	such that 
	$$ g \in C([0,T]; A^0_{\nu\ast} (\TT) )  $$
	is a solution of \eqref{eq:grafo2} in  the RT unstable regime,
	$$
	\rho^--\rho^+<0,
	$$
	and 
	\begin{equation} 
		\norm{g(\tau)}_{A^0_{\nu^\ast}}  \ge \frac{1}{C}	e^{\sqrt{(\seminorm{\rho^--\rho^+}/4) \nu^\ast \tau}}	\norm{g_0}_{{A}^0} \quad \tau \in [0,T] .
	\end{equation}
\end{theo}

\begin{proof}
	
Let $T>0$ arbitrary. By Theorem \ref{exponentialdecay}, there exists a solution for every $\tau \in [0,T]$
	$$h \in C([0,\tau];{A}^0_{\nu^\ast})$$
 in the RT stable case starting from a small enough initial data $h_0$,
		$$ \norm{h_0}_{A^0_{\nu^\ast}} < \varepsilon.$$
Setting a family of different initial data $h_0$, we can find a family of different solutions $h(\alpha,t)$. Furthermore, each solution satisfies 
\begin{equation} \label{stabexp}
e^{\sqrt{\seminorm{\bar \rho} \nu^\ast t}}	\norm{h(t)}_{{A}^0} \leq C \norm{h_0}_{A^0_{\nu^\ast}}  \quad t \in [0,\tau] 
\end{equation}
where $C >0$.

Now, define the function
\begin{equation*}
	g(\alpha,t) = h(\alpha,\tau-t).
\end{equation*}
It is clear from the definition that 
	$$g \in C([0,\tau];{A}^0_{\nu^\ast}).$$
Moreover, we find that 
\begin{align*}
	g_t(\alpha,t) &= -(h_t)(\alpha,\tau-t) \\ &= |\bar \rho| \Lambda^{-1} (g)(\alpha,t) - \frac{|\bar \rho| }{2 \pi} \left( I_1(g)(\alpha,t) + I_2(g)(\alpha,t) + I_3(g)(\alpha,t) + I_4(g)(\alpha,t)   \right)	,
\end{align*}
i.e., $g(\alpha,t)$ solves \eqref{eq:grafo2} (equivalently \eqref{isolatelinear}) in the unstable RT scenario. Note that 
$$g_0(\alpha) = h(\alpha,\tau) \text{ and } g(\alpha,\tau) = h_0(\alpha).$$
Then, by \eqref{stabexp}, it holds that 
\begin{equation}\label{expgrowth} 
\norm{g(\tau)}_{A^0_{\nu^\ast}}  \ge \frac{1}{C}	e^{\sqrt{\seminorm{\bar \rho} \nu^\ast \tau}}	\norm{g_0}_{{A}^0}  \quad \tau \in [0,T].
\end{equation}
The previous equation shows exponential growth of smooth analytic solutions of the unstable RT scenario for small initial data, in a time interval which is arbitrarily large.

\end{proof}

\subsection*{Acknowledgments.} 		
F. Gancedo and E. Salguero were partially supported by the ERC through the Starting Grant H2020-EU.1.1.-639227, by the MICINN (Spain) through the grants EUR2020-112271 and PID2020-114703GB-I00 and by the Junta de Andalucía through the grant P20-00566. F. Gancedo was partially supported by MINECO grant RED2018-102650-T (Spain). F. Gancedo  acknowledges support
from IMAG, funded by MICINN through the Maria de Maeztu
Excellence Grant CEX2020-001105-M/AEI/10.13039/501100011033. R. Granero-Belinchón was supported by the project ``Mathematical Analysis of Fluids and Applications" Grant PID2019-109348GA-I00 funded by MCIN/AEI/ 10.13039/501100011033 and acronym "MAFyA". This publication is part of the project PID2019-109348GA-I00/AEI/10.13039/501100011033. R. Granero-Belinchón is also supported by a 2021 Leo\-nar\-do Grant for Researchers and Cultural Creators, BBVA Foundation. The BBVA Foundation accepts no responsibility for the opinions, statements, and contents included in the project and/or the results thereof, which are entirely the responsibility of the authors. E. Salguero was partially supported by the grant PRE2018-083984, funded by MCIN/AEI/ 10.13039/501100011033.

\bibliographystyle{abbrv}

\begin{thebibliography}{10}

\bibitem{giga_acta}
K.~Abe and Y.~Giga.
\newblock Analyticity of the {S}tokes semigroup in spaces of bounded functions.
\newblock {\em Acta Mathematica}, 211(1):1--46, 2013.

\bibitem{free_boundary}
S.~Antontsev, A.~Meirmanov and B. V. Yurinsky.
\newblock A free-boundary problem for Stokes equations: classical solutions.
\newblock {\em Interfaces and Free Boundaries}, 2:413--424, 2000.

\bibitem{badea_capillary_1998}
A.~Badea and J.~Duchon.
\newblock Capillary driven evolution of an interface between viscous fluids.
\newblock {\em Nonlinear Analysis: Theory, Methods \& Applications},
  31(3-4):385--403, Feb. 1998.

\bibitem{bertozzi_global_1993}
A.~L. Bertozzi and P.~Constantin.
\newblock Global regularity for vortex patches.
\newblock {\em Communications in Mathematical Physics}, 152(1):19--28, Feb.
  1993.

\bibitem{castro2013finite}
A.~Castro, D.~C{\'o}rdoba, C.~Fefferman, F.~Gancedo, and J.~G{\'o}mez-Serrano.
\newblock Finite time singularities for the free boundary incompressible
  {Euler} equations.
\newblock {\em Annals of Mathematics}, pages 1061--1134, 2013.

\bibitem{castro2013breakdown}
{\'A}.~Castro, D.~C{\'o}rdoba, C.~Fefferman, and F.~Gancedo.
\newblock Breakdown of smoothness for the {Muskat} problem.
\newblock {\em Archive for Rational Mechanics and Analysis}, 208(3):805--909,
  2013.

\bibitem{castro2012rayleigh}
{\'A}.~Castro, D.~C{\'o}rdoba, C.~Fefferman, F.~Gancedo, and
  M.~L{\'o}pez-Fern{\'a}ndez.
\newblock Rayleigh-Taylor breakdown for the {M}uskat problem with applications to
  water waves.
\newblock {\em Annals of mathematics}, pages 909--948, 2012.

\bibitem{castro_uniformly_2016}
A.~Castro, D.~Córdoba, and J.~Gómez-Serrano.
\newblock Uniformly {Rotating} {Analytic} {Global} {Patch} {Solutions} for
  {Active} {Scalars}.
\newblock {\em Annals of PDE}, 2(1):1, June 2016.

\bibitem{chemin_persistance_1993}
J.-Y. Chemin.
\newblock Persistance de structures géométriques dans les fluides
  incompressibles bidimensionnels.
\newblock {\em Annales scientifiques de l'École normale supérieure},
  26(4):517--542, 1993.

\bibitem{chen_peskin_2021}
K.~Chen and Q.-H. Nguyen.
\newblock The {Peskin} problem with $\text{B}_{\infty,\infty}^1$ initial data,
  Dec. 2021.
\newblock arXiv:2107.13854 [math].

\bibitem{cheng_well-posedness_2016}
C.~A. Cheng, R.~Granero-Belinchón, and S.~Shkoller.
\newblock Well-posedness of the {Muskat} problem with ${H}^2$ initial data.
\newblock {\em Advances in Mathematics}, 286:32--104, Jan. 2016.

\bibitem{cordoba_interface_2010}
A.~Córdoba, D.~Córdoba, and F.~Gancedo.
\newblock Interface evolution: {Water} waves in 2-{D}.
\newblock {\em Advances in Mathematics}, 223(1):120--173, Jan. 2010.

\bibitem{cordoba_interface_2011}
A.~Córdoba, D.~Córdoba, and F.~Gancedo.
\newblock Interface evolution: the {Hele}-{Shaw} and {Muskat} problems.
\newblock {\em Annals of Mathematics}, 173(1):477--542, Jan. 2011.

\bibitem{cordoba_contour_2007}
D.~Córdoba and F.~Gancedo.
\newblock Contour {Dynamics} of {Incompressible} 3-{D} {Fluids} in a {Porous}
  {Medium} with {Different} {Densities}.
\newblock {\em Communications in Mathematical Physics}, 273(2):445--471, July
  2007.

\bibitem{elgindi_asymptotic_2017}
T.~M. Elgindi.
\newblock On the {Asymptotic} {Stability} of {Stationary} {Solutions} of the
  {Inviscid} {Incompressible} {Porous} {Medium} {Equation}.
\newblock {\em Archive for Rational Mechanics and Analysis}, 225(2):573--599,
  Aug. 2017.

\bibitem{gancedo_survey_2017}
F.~Gancedo.
\newblock A survey for the {Muskat} problem and a new estimate.
\newblock {\em SeMA Journal}, 74(1):21--35, Mar. 2017.

\bibitem{gancedo_global_2021}
F.~Gancedo, E.~Garcia-Juarez, N.~Patel, and R.~Strain.
\newblock Global {Regularity} for {Gravity} {Unstable} {Muskat} {Bubbles}, June
  2021.
\newblock arXiv:1902.02318 [math].

\bibitem{gancedo_well-posedness_2021}
F.~Gancedo, H.~Q. Nguyen, and N.~Patel.
\newblock Well-posedness for {SQG} sharp fronts with unbounded curvature, May
  2021.
\newblock arXiv:2105.10982 [math].

\bibitem{gancedo2021local}
F.~Gancedo and N.~Patel.
\newblock On the local existence and blow-up for generalized {SQG} patches.
\newblock {\em Annals of PDE}, 7(1):1--63, 2021.

\bibitem{garcia-juarez_peskin_2020}
E.~Garcia-Juarez, Y.~Mori, and R.~M. Strain.
\newblock The {Peskin} {Problem} with {Viscosity} {Contrast}, Oct. 2020.
\newblock arXiv:2009.03360 [math].

\bibitem{granero2022interfaces}
R.~Granero-Belinch{\'o}n.
\newblock Interfaces in incompressible flows.
\newblock {\em SeMA Journal}, pages 1--25, 2022.

\bibitem{granero2020growth}
R.~Granero-Belinch{\'o}n and O.~Lazar.
\newblock Growth in the {M}uskat problem.
\newblock {\em Mathematical Modeling of Natural Phenomena}, 15:7, 2020.

\bibitem{granero-belinchon_well-posedness_2021}
R.~Granero-Belinchón and S.~Scrobogna.
\newblock Well-posedness of the water-wave with viscosity problem.
\newblock {\em Journal of Differential Equations}, 276:96--148, Mar. 2021.

\bibitem{grayer_ii_dynamics_2022}
H.~Grayer~II.
\newblock Dynamics of density patches in infinite {Prandtl} number convection,
  July 2022.
\newblock arXiv:2207.09738 [physics].

\bibitem{Guo2007}
Y.~Guo, C.~Hallstrom and D. Spirn.
\newblock Dynamics near unstable, interfacial fluids.
\newblock {\em Communications in Mathematical Physics}, 270(3):635--689, 2007.

\bibitem{guo_almost_2013}
Y.~Guo and I.~Tice.
\newblock Almost {Exponential} {Decay} of {Periodic} {Viscous} {Surface}
  {Waves} without {Surface} {Tension}.
\newblock {\em Archive for Rational Mechanics and Analysis}, 207(2):459--531,
  Feb. 2013.

\bibitem{guo2018stability}
Y.~Guo and I.~Tice.
\newblock Stability of contact lines in fluids: 2d {Stokes} flow.
\newblock {\em Archive for Rational Mechanics and Analysis}, 227(2):767--854,
  2018.

\bibitem{hofer_2018_stokes}
R.~M. H\"{o}fer.
\newblock Sedimentation of inertialess particles in {S}tokes flows.
\newblock  {\em Communications in Mathematical Physics}, 360(1):55--101, 2018.

\bibitem{hofer_2022_filaments}
R.~M. H\"{o}fer, C.~Prange, and F.~Sueur.
\newblock {M}otion of several slender rigid filaments in a {S}tokes flow.
\newblock {\em Journal de l'\'{E}cole polytechnique. Math\'{e}matiques}, 9:327--380, 2022.

\bibitem{hofer_2021_stokes}
R.~M. H\"{o}fer and R.~Schubert.
\newblock The influence of {E}instein's effective viscosity on sedimentation at
  very small particle volume fraction.
\newblock {\em Annales de l'Institut Henri Poincar\'{e} C. Analyse Non Lin\'{e}aire},
  38(6):1897--1927, 2021.
  
  \bibitem{Kiselev2019}
  A.~Kiselev and C.~Li.
  \newblock Global regularity and fast small-scale formation for {E}uler
  patch equation in a smooth domain.
  \newblock {\em Communications in Partial Differential Equations}, 44(4):279--308, 2019.

\bibitem{kiselev2016finite}
A.~Kiselev, L.~Ryzhik, Y.~Yao, and A.~Zlato{\v{s}}.
\newblock Finite time singularity for the modified {SQG} patch equation.
\newblock {\em Annals of mathematics}, pages 909--948, 2016.

\bibitem{kiselev2017local}
A.~Kiselev, Y.~Yao, and A.~Zlato{\v{s}}.
\newblock Local regularity for the modified {SQG} patch equation.
\newblock {\em Communications on Pure and Applied Mathematics},
  70(7):1253--1315, 2017.

\bibitem{ladyzenskaja_1980}
O.~A. Lady\v{z}enskaja and V.~A. Solonnikov.
\newblock Determination of solutions of boundary value problems for stationary
  {S}tokes and {N}avier-{S}tokes equations having an unbounded {D}irichlet
  integral.
\newblock {\em Zap. Nauchn. Sem. Leningrad. Otdel. Mat. Inst. Steklov. (LOMI)},
  96:117--160, 308, 1980.
\newblock Boundary value problems of mathematical physics and related questions
  in the theory of functions, 12.

\bibitem{ladyzhenskaya_mathematical_1969}
O.~A. Ladyzhenskaya.
\newblock {\em The mathematical theory of viscous incompressible flow}.
\newblock Mathematics and its {Applications}, {Vol}. 2. Gordon and Breach
  Science Publishers, New York-London-Paris, 1969.

\bibitem{leblond}
A.~Leblond.
\newblock Well-posedness of the {S}tokes-transport system in bounded domains
  and in the infinite strip.
\newblock {\em Journal de Math\'{e}matiques Pures et Appliqu\'{e}es. Neuvi\`eme S\'{e}rie}, 158:120--143, 2022.

\bibitem{lin2019solvability}
F.-H. Lin and J.~Tong.
\newblock Solvability of the {Stokes} immersed boundary problem in two
  dimensions.
\newblock {\em Communications on Pure and Applied Mathematics}, 72(1):159--226,
  2019.

\bibitem{longuet1992theory}
M.~S. Longuet-Higgins.
\newblock Theory of weakly damped {Stokes} waves: a new formulation and its
  physical interpretation.
\newblock {\em Journal of Fluid Mechanics}, 235:319--324, 1992.

\bibitem{majda_vorticity_2001}
A.~J. Majda and A.~L. Bertozzi.
\newblock {\em Vorticity and {Incompressible} {Flow}}.
\newblock Cambridge University Press, 1 edition, Nov. 2001.

\bibitem{matioc_two-phase_2021}
B.~Matioc and G.~Prokert.
\newblock Two-phase {Stokes} flow by capillarity in full {2D} space: an
  approach via hydrodynamic potentials.
\newblock {\em Proceedings of the Royal Society of Edinburgh: Section A
  Mathematics}, 151(6):1815--1845, Dec. 2021.

\bibitem{matioc_two-phase_2022}
B.~Matioc and G.~Prokert.
\newblock Two-phase {Stokes} flow by capillarity in the plane: {The} case of
  different viscosities.
\newblock {\em Nonlinear Differential Equations and Applications},
  29(5):54, Sept. 2022.

\bibitem{matioc_capillarity_2022}
B.-V. Matioc and G.~Prokert.
\newblock Capillarity driven {S}tokes flow: the one-phase problem as small
  viscosity limit, Sept. 2022.
\newblock arXiv:2209.13376 [math].

\bibitem{mecherbet_2019_stokes}
A.~Mecherbet.
\newblock Sedimentation of particles in {S}tokes flow.
\newblock {\em Kinetic and Related Models}, 12(5):995--1044, 2019.

\bibitem{mecherbet_2021_droplet}
A.~Mecherbet.
\newblock On the sedimentation of a droplet in {S}tokes flow.
\newblock {\em Communications in Mathematical Sciences}, 19(6):1627--1654, 2021.

\bibitem{mecherbet_few_2022}
A.~Mecherbet and F.~Sueur.
\newblock A few remarks on the transport-{S}tokes system, Sept. 2022.
\newblock arXiv:2209.11637 [math].

\bibitem{mori_wellposedness_2019}
Y.~Mori, A.~Rodenberg, and D.~Spirn.
\newblock Well‐{Posedness} and {Global} {Behavior} of the {P}eskin {Problem}
  of an {Immersed} {Elastic} {Filament} in {S}tokes {Flow}.
\newblock {\em Communications on Pure and Applied Mathematics}, 72(5):887--980,
  May 2019.

\bibitem{zheng2017local}
Y.~Zheng and I.~Tice.
\newblock Local well posedness of the near-equilibrium contact line problem in
  2-dimensional {S}tokes flow.
\newblock {\em SIAM Journal on Mathematical Analysis}, 49(2):899--953, 2017.

\end{thebibliography}

\end{document}